\newif \ifJournal \Journalfalse

\ifJournal

\documentclass{ws-m3as}
\usepackage[super]{cite}
\usepackage{url}

\else

\documentclass[a4paper,11pt]{article}
\usepackage[english]{babel}
\usepackage[utf8]{inputenc}
\usepackage[T1]{fontenc}

\fi

\usepackage{amsmath,amsfonts,amssymb,mathtools}
\usepackage{mathbbol}
\DeclareSymbolFontAlphabet{\mathbb}{AMSb}
\DeclareSymbolFontAlphabet{\mathbbl}{bbold}
\DeclareMathAlphabet{\mathbbmsl}{U}{bbm}{m}{sl}
\usepackage{upgreek}
\usepackage{nicefrac}
\usepackage{cancel}
\usepackage{comment}
\usepackage{enumitem}
\usepackage{physics}
\usepackage{graphicx}
\usepackage[dvipsnames, svgnames]{xcolor}
\usepackage{pgfplots}
\usetikzlibrary{patterns}
\usepackage{multirow}
\usepackage{subcaption}
\usepackage{etoolbox}
\usepackage[skins]{tcolorbox}
\makeatletter
\patchcmd{\ttlh@hang}{\parindent\z@}{\parindent\z@\leavevmode}{}{}
\patchcmd{\ttlh@hang}{\noindent}{}{}{}
\makeatother

\usepackage{hyperref}
\hypersetup{
  colorlinks = true,                        
  linkcolor = OliveGreen,                   
  citecolor = violet,                       
  filecolor = blue,                         
  urlcolor = magenta                        
}

\ifJournal

\catchline{}{}{}{}{}
\title{Discrete Weber inequalities and related Maxwell compactness for hybrid spaces over polyhedral partitions of domains with general topology}
\author{Simon Lemaire}
\address{Inria, Univ.~Lille, CNRS\\ UMR 8524 -- Laboratoire Paul Painlev\'e\\ 59000 Lille, France\\\email{simon.lemaire@inria.fr}}
\author{Silvano Pitassi}
\address{Inria, Univ.~Lille, CNRS\\ UMR 8524 -- Laboratoire Paul Painlev\'e\\ 59000 Lille, France\\\email{silvano.pitassi@inria.fr}}

\else

\usepackage[hmargin={28mm,28mm},vmargin={30mm,35mm}]{geometry}
\usepackage{amsthm}
\usepackage{authblk}
\usepackage{newtxtext,newtxmath}
\newtheorem{theorem}{Theorem}
\newtheorem{lemma}[theorem]{Lemma}
\newtheorem{proposition}[theorem]{Proposition}

\newtheorem{remark}[theorem]{Remark}
\newcommand{\email}[1]{\href{mailto:#1}{#1}}

\title{Discrete Weber inequalities and related Maxwell compactness for hybrid spaces over polyhedral partitions of domains with general topology}
\author[1]{Simon Lemaire\footnote{\email{simon.lemaire@inria.fr} (corresponding author)}}
\affil[1]{Inria, Univ.~Lille, CNRS, UMR 8524 -- Laboratoire Paul Painlev\'e, 59000 Lille, France}
\author[1]{Silvano Pitassi\footnote{\email{silvano.pitassi@inria.fr}}}

\fi

\newlength{\arrow}
\newcommand{\defi}{\mathrel{\mathop:}=}
\newcommand{\ifed}{=\mathrel{\mathop:}}
\newcommand{\st}{~:~}

\renewcommand{\vec}[1]{\boldsymbol{#1}}
\newcommand{\normal}{\vec{n}}

\newcommand{\Natu}{\mathbb{N}}
\newcommand{\Real}{\mathbb{R}}
\newcommand{\Comp}{\mathbb{C}}

\DeclareMathOperator{\Div}{div}
\DeclareMathOperator{\Curl}{\bf curl}
\DeclareMathOperator{\Grad}{\bf grad}

\newcommand{\LL}[1][\Omega]{L^2(#1)}
\newcommand{\bLL}[1][\Omega]{\vec{L}^2(#1)}
\newcommand{\Hdiv}[1][\Omega]{\vec{H}(\Div_{\eta} ; #1)}
\newcommand{\Hzdiv}[1][\Omega]{\vec{H}_0(\Div_{\eta} ; #1)}
\newcommand{\Hdivz}[1][\Omega]{\vec{H}(\Div^0_{\eta} ; #1)}
\newcommand{\Hzdivz}[1][\Omega]{\vec{H}_0(\Div^0_{\eta} ; #1)}
\newcommand{\Hcurl}[1][\Omega]{\vec{H}(\Curl ; #1)}
\newcommand{\Hzcurl}[1][\Omega]{\vec{H}_{\vec{0}}(\Curl ; #1)}
\newcommand{\Hcurlz}[1][\Omega]{\vec{H}(\Curl^{\vec{0}} ; #1)}
\newcommand{\Hzcurlz}[1][\Omega]{\vec{H}_{\vec{0}}(\Curl^{\vec{0}} ; #1)}

\newcommand{\cweber}{\vec{X}(\Omega)}
\newcommand{\cfweber}{\vec{X}_{\vec{\tau}}(\Omega)}
\newcommand{\csweber}{\vec{X}_{n}(\Omega)}

\newcommand{\fharmonic}{\vec{H}_{\vec{\tau}}(\Omega)}
\newcommand{\sharmonic}{\vec{H}_{n}(\Omega)}

\newcommand{\Poly}{\mathcal{P}}
\newcommand{\bPoly}{\vec{\Poly}}

\newcommand{\pGT}[1]{\vec{\mathcal{G}}^{#1}(T)}
\newcommand{\kGT}[1]{\vec{\mathcal{G}}^{{\rm c},#1}(T)}
\newcommand{\pRT}[1]{\vec{\mathcal{R}}^{#1}(T)}
\newcommand{\kRT}[1]{\vec{\mathcal{R}}^{{\rm c},#1}(T)}

\newcommand{\pRF}[1]{\vec{\mathcal{R}}^{#1}(F)}
\newcommand{\kRF}[1]{\vec{\mathcal{R}}^{{\rm c},#1}(F)}
\newcommand{\qRF}[1]{\vec{\mathcal{Q}}^{#1}(F)}

\newcommand{\Hindex}{\mathcal{H}}
\renewcommand{\t}{\mathcal{T}}
\newcommand{\f}{\mathcal{F}}

\newcommand{\dweber}{\underline{\mathbbmsl{X}}_h^{\ell}}
\newcommand{\dweberT}{\underline{\mathbbmsl{X}}_T^{\ell}}

\newcommand{\dfweber}{\underline{\mathbbmsl{X}}_{h,\vec{\tau}}^{\ell}}
\newcommand{\dsweber}{\underline{\mathbbmsl{X}}_{h,n}^{\ell}}

\newcommand{\dweberc}{\underline{\vec{X}}_h^{\ell}}
\newcommand{\dfweberc}{\underline{\vec{X}}_{h,\vec{\tau}}^{\ell}}

\newcommand{\vh}{\underline{\mathbbmsl{v}}_h}
\newcommand{\vhc}{\underline{\vec{v}}_h}
\newcommand{\vTh}{\vec{v}_h}

\newcommand{\vTF}{\underline{\mathbbmsl{v}}_T}
\newcommand{\vT}{\vec{v}_T}
\newcommand{\vFt}{\vec{v}_{F,\vec{\tau}}}
\newcommand{\vFn}{[\eta v]_{F,n}}

\newcommand{\wTF}{\underline{\mathbbmsl{w}}_T}

\newcommand{\wFt}{\vec{w}_{F,\vec{\tau}}}
\newcommand{\wFn}{[\eta w]_{F,n}}

\newcommand{\DpT}{D_T^p}
\newcommand{\Dph}{D_h^p}

\newcommand{\CmT}{\vec{C}_T^m}
\newcommand{\Cmh}{\vec{C}_h^m}

\begin{document}

\maketitle

\ifJournal
\else

\begin{abstract}
  We prove discrete versions of the first and second Weber inequalities on $\vec{H}(\Curl)\cap\vec{H}(\Div_{\eta})$-like hybrid spaces spanned by polynomials attached to the faces and to the cells of a polyhedral mesh. The proven hybrid Weber inequalities are optimal in the sense that (i) they are formulated in terms of $\vec{H}(\Curl)$- and $\vec{H}(\Div_{\eta})$-like hybrid semi-norms designed so as to embed optimally (polynomially) consistent face penalty terms, and (ii) they are valid for face polynomials in the smallest possible stability-compatible spaces. Our results are valid on domains with general, possibly non-trivial topology. In a second part we also prove, within a general topological setting, related discrete Maxwell compactness properties.
\end{abstract}

\fi

\section{Introduction}

Let $\Omega$ be a domain in $\Real^3$, i.e.~a bounded and connected Lipschitz open set of $\Real^3$.
Our main motivation in this work comes from the study of time-harmonic Maxwell's equations. We are interested in PDE models, posed in $\Omega$, for which the main (vector) unknown $\vec{u}$ lies in a subspace $\vec{X}_\star(\Omega)$ of the space
$$\cweber\defi\Hcurl\cap\Hdiv,$$
where $\eta$ is a physical parameter (possibly varying inside $\Omega$) and $\Div_\eta\defi\Div(\eta\,\cdot)$. In practice, the unknown $\vec{u}$ is the (vector) complex amplitude of a time-harmonic electromagnetic field, which may either be the electric field $\vec{e}$ (then $\eta$ is the electric permittivity $\varepsilon$), or the magnetic field $\vec{h}$ (then $\eta$ is the magnetic permeability $\mu$), or some vector potential $\vec{a}$ (then $\eta\equiv 1$ if one adopts the Coulomb gauge). Vector fields in $\vec{X}_\star(\Omega)$ bear, almost everywhere on the boundary of $\Omega$, either zero tangential or zero $\eta$-weighted normal trace. When zero tangential trace is imposed on the whole boundary of $\Omega$, we let $\vec{X}_\star(\Omega)\ifed\cfweber$, and we refer to this case as fully tangential. In the same vein, when zero $\eta$-weighted normal trace is imposed on the whole boundary of $\Omega$, we let $\vec{X}_\star(\Omega)\ifed\csweber$, and we refer to this case as fully normal. All other cases are referred to as mixed.

It is nowadays well-understood that the well-posedness of variational problems set in $\vec{X}_\star(\Omega)$ is topology-depending, and directly relates to the (co)homology of the underlying de Rham complex.
For standard (i.e.~either fully tangential, or fully normal) boundary conditions, the homology of the de Rham complex has been extensively studied in the literature, and the homology spaces, spanned by harmonic vector fields, have been fully characterized~\cite{GiRav:86,DaLio:90,ABDGi:98,GroKo:04}. Based on the latter characterizations, topologically general Poincar\'e-type inequalities, the so-called {\em Weber inequalities}, can be established. These inequalities are named after Christian Weber, in relation to his seminal contributions to the topic~\cite{Weber:80}. The first Weber inequality corresponds to the fully tangential case of vector fields in $\cfweber$, whereas the second corresponds to the fully normal case (of vector fields in $\csweber$). Whenever $\Omega$ has trivial topology (think of a contractible domain), letting $\|{\cdot}\|_0$ indifferently denote the $\{\LL,\bLL\}$-norm, the first and second Weber inequalities read:
\begin{equation} \label{eq:example}
  \|\vec{v}\|_0\leq C_{\eta,\star}\big(\|\Div(\eta\vec{v})\|_0+\|\Curl\vec{v}\|_0\big)\qquad\forall\vec{v}\in\vec{X}_\star(\Omega),
\end{equation}
for a real number $C_{\eta,\star}>0$ only depending on $\Omega$, $\eta$, and on the boundary conditions type.
Weber inequalities are closely related to the compact embedding of $\vec{X}_\star(\Omega)$ into $\bLL$, also known in the literature under the name of {\em Maxwell compactness} property (see e.g.~\cite{BaPaS:16}). Weber inequalities must not be confused with Gaffney inequalities, which are related to the continuous embedding of $\vec{X}_\star(\Omega)$ into $\vec{H}^1(\Omega)$, only valid for smooth $\eta$ and smooth or convex domains (see~\cite[Sect.~2.3]{ABDGi:98} for further details). On general Lipschitz domains $\Omega$, even under the additional assumption that $\Omega$ is a polyhedron, one can solely prove that $\vec{X}_\star(\Omega)$ is continuously embedded (i) into $\vec{H}^{\frac12+s}(\Omega)$ for some $s>0$ if $\eta$ is globally smooth (cf.~\cite{Costa:90} and~\cite[Prop.~3.7]{ABDGi:98}), or (ii) into $\vec{H}^s(\Omega)$ (still for some $s>0$) if $\eta$ is piecewise smooth (see~\cite{CoDaN:99,BGLud:13}). When departing from standard boundary conditions, the (co)homology of the de Rham complex can be more complicated. For mixed boundary conditions, the characterization of the homology spaces has been investigated in~\cite{FerGi:97}, $\vec{H}^s(\Omega)$-regularity for piecewise smooth $\eta$ proven in~\cite{Jochm:99}, and Maxwell compactness in~\cite{Jochm:97}.  

Let us briefly comment at this point on the terminology. We endow $\cweber$ with its standard norm $\|\vec{v}\|_{\vec{X}}^2\defi\left(\|\vec{v}\|_0^2+\|\Curl\vec{v}\|_0^2\right)+\left(\|\vec{v}\|_0^2+\|\Div(\eta\vec{v})\|_0^2\right)$.
From an historical perspective, what Weber actually proved in~\cite{Weber:80} (based on a preliminary work by Weck~\cite{Weck:74}) are not exactly inequalities of the form~\eqref{eq:example}, but rather Maxwell compactness properties in $\left(\vec{X}_\star(\Omega),\|{\cdot}\|_{\vec{X}}\right)$. Whereas the $\|{\cdot}\|_{\vec{X}}$-norm embeds an $\bLL$-norm control on the function, at the opposite, remark that such a control is absent from the right-hand side of~\eqref{eq:example}. To prove Poincar\'e-type inequalities of the form~\eqref{eq:example} from Maxwell compactness, we know (by Peetre--Tartar's lemma) that it is sufficient to identify, for some Banach space $\vec{\mathcal{Y}}$, an injective bounded linear mapping $\vec{\mathcal{A}}:\left(\vec{X}_\star(\Omega),\|{\cdot}\|_{\vec{X}}\right)\to\left(\vec{\mathcal{Y}},\|{\cdot}\|_{\vec{\mathcal{Y}}}\right)$ such that
$$\|\vec{v}\|_{\vec{X}}\leq C_1\big(\|\vec{\mathcal{A}}(\vec{v})\|_{\vec{\mathcal{Y}}}+\|\vec{v}\|_0\big)\qquad\forall\vec{v}\in\vec{X}_\star(\Omega).$$
It then follows that
$$\|\vec{v}\|_0\leq\|\vec{v}\|_{\vec{X}}\leq C_2\|\vec{\mathcal{A}}(\vec{v})\|_{\vec{\mathcal{Y}}}\qquad\forall\vec{v}\in\vec{X}_\star(\Omega).$$
For instance, when $\Omega$ has trivial topology, the mapping
$$\vec{\mathcal{A}}:\vec{X}_\star(\Omega)\to\vec{\mathcal{Y}}\defi\LL\times\bLL;\vec{v}\mapsto\big(\Div(\eta\vec{v}),\Curl\vec{v}\big)$$
satisfies all the required properties, whence the inequality~\eqref{eq:example}. Within a general topological setting, identifying $\vec{\mathcal{A}}$ crucially requires a precise characterization of the (co)homology spaces of the de Rham complex. Such a characterization has not been provided by Weber in his seminal work, but by different authors later on (cf.~e.g.~\cite{GiRav:86,DaLio:90,ABDGi:98,GroKo:04}). In this respect, naming Poincar\'e-type inequalities of the form~\eqref{eq:example} ``Weber inequalities'' is imperfect. However, considering that Weber was the first key player in their discovery, and following~\cite{ACJLa:18}, we will adopt this terminology throughout this article.

\medskip

The focus in this work is on (arbitrary-order) {\em skeletal methods} over polyhedral partitions. Contrary to plain vanilla Discontinuous Galerkin (DG) methods, skeletal methods attach unknowns to the mesh skeleton, thus allowing for the static condensation of (potential) cell unknowns (cf.~\cite{Lemai:21} for further insight).
In what follows, for $\mathcal{X}(\Omega)$ a given functional space on $\Omega$, we will term $\mathcal{X}(\Omega)$-conforming any skeletal method attaching unknowns to the boundary of the cells so as to mimic the ``continuity'' properties of the underlying continuous space $\mathcal{X}(\Omega)$. For instance, in the lowest-order case, $H^1(\Omega)$-conforming methods attach unknowns to the mesh vertices so as to emulate (full) continuity of traces, whereas $\Hcurl$-conforming methods attach unknowns to the mesh edges so as to reproduce tangential continuity of (vector) traces. As opposed, we will term non-conforming any other kind of skeletal method. In the polyhedral context, this dichotomy is particularly convenient to discriminate between the approaches, for which, anyhow at the end, only fully discontinuous polynomial proxies are computable.
We exclusively focus in this work on non-conforming (skeletal) methods, often referred to in the literature as {\em hybrid methods}. Hybrid methods only attach unknowns to the faces and to the cells of the mesh.
Examples of such approaches include the Hybridizable Discontinuous Galerkin (HDG)~\cite{CoGoL:09}, the Weak Galerkin (WG)~\cite{WanYe:13}, the Hybrid High-Order (HHO)~\cite{DPErn:15,DPELe:14}, or the non-conforming Virtual Element (nc-VE)~\cite{AdDLM:16} methods.
All these technologies share tight links, as first observed in~\cite{CDPEr:16}.
At the opposite side of the skeletal spectrum lie conforming polyhedral methods, whose salient examples include the conforming Virtual Element (c-VE)~\cite{BdVBC:13} or the Discrete De Rham (DDR)~\cite{DPDro:23} approaches, between which strong connections also exist.

In this work, we aim at deriving, within a general topological setting, discrete functional analysis tools for hybrid approximations of div-curl systems over polyhedral mesh families.
To our knowledge, hybrid methods have only been studied yet within a trivial topological setting.
In the HDG and WG contexts, these contributions respectively include~\cite{NPCoc:11,CQSSo:17,CCuXu:19} (cf.~also~\cite{DuSay:20}) and~\cite{MWYZh:15}. In the HHO context, a discrete (first) Weber inequality was proven in~\cite{CDPLe:22} for $\vec{H}(\Curl)$-like hybrid spaces satisfying a discrete divergence-free constraint in trivial domains, and leveraged to perform the analysis of HHO approximations of first- and second-order magnetostatics models.
In this paper, building on~\cite{CDPLe:22}, we prove discrete versions of the first and second Weber inequalities on generic $\vec{H}(\Curl)\cap\vec{H}(\Div_\eta)$-like hybrid spaces, which are valid for domains with general, possibly non-trivial topology.
We also establish the related discrete Maxwell compactness properties.
The results are valid for complex-valued fields and piecewise constant (over the polyhedral partition), real-valued, isotropic parameters $\eta$.
The results seamlessly extend to the case of complex-valued, Hermitian anisotropic parameters.
The hybrid Weber inequalities we prove are optimal in the following sense:
\begin{itemize}
  \item[(i)] they are formulated in terms of $\vec{H}(\Curl)$- and $\vec{H}(\Div_{\eta})$-like hybrid semi-norms designed so as to embed optimally (polynomially) consistent face penalty terms;
  \item[(ii)] they are valid for face polynomials in the smallest possible stability-compatible spaces.
\end{itemize}
The first property above guarantees that the corresponding stabilization Hermitian forms, of the so-called Lehrenfeld--Sch\"oberl type~\cite{LeSch:16} (cf.~also~\cite{CEPig:21} for further insight), shall provide superconvergence when used within hybrid numerical schemes (cf.~\cite{CQSSo:17,CDPLe:22}).
The second property, in turn, allows for substantial savings when considering the approximation of first-order systems, for which discrete differential operators need not be reconstructed (cf.~\cite[Sect.~3.1]{CDPLe:22}).
More generally, our work paves the way to the analysis of hybrid polyhedral approximations of general div-curl systems on domains with arbitrary topology, including non-linear (Hilbertian) problems as they may arise e.g.~in the modeling of ferromagnetic materials.

Let us finally relate our work to the so-called compatible approaches.
In the Finite Element (FE) context over matching tetrahedral partitions, N\'ed\'elec elements~\cite{Nedel:80,Nedel:86} can be used to define discrete (polynomial) de Rham complexes~\cite{Bossa:98,Monk:03}. Through the prism of Whitney forms~\cite{Whitn:57} and, more broadly, of the Finite Element Exterior Calculus (FEEC)~\cite{AFaWi:06,Arnol:18}, the notion of discrete de Rham complex can be extended from the framework of vector calculus to the one of exterior calculus over differential forms. The corresponding numerical methods then lie in an analytical setting which is referred to as {\em compatible}.
In a compatible FE setting, Poincar\'e-type inequalities are straightforward consequences of the (co)homological properties of the underlying polynomial de Rham complex. In the polyhedral context, following the DDR approach~\cite{DPDro:23,BDPDH:23}, discrete de Rham complexes can also be written and analyzed, but at the price of some complications. These complications include the use of discrete (reconstructed) differential operators over algebraic spaces of discrete unknowns, and the tedious tracking in the analysis of mesh-dependent quantities.
Notice that all the above-mentioned compatible approaches are related to conforming (FE or polyhedral) methods.
For non-conforming methods, the only contributions we are aware of towards the developement of a compatible analytical setting are~\cite{Licht:17,ChLic:20} (based on seminal ideas from~\cite{BraSc:08}), in which the convenient notion of discrete distributional differential form is introduced. The framework considered in~\cite{Licht:17,ChLic:20} is the one of DG spaces over matching simplicial meshes.
The subject thus still seems largely open for extending such ideas to hybrid spaces over polyhedral partitions.
We believe our work may constitute, in the language of vector calculus, a first step in this direction.

\medskip

The rest of the article is organized as follows. In Section~\ref{se:prelim}, after introducing both the topological and functional frameworks, we state and prove, based on ad hoc Helmholtz--Hodge decompositions, the continuous first and second Weber inequalities in $\Hcurl\cap\Hdiv$, as well as the related Maxwell compactness properties. In Section~\ref{se:dis.set} we introduce the discrete setting, and define the notions of polyhedral mesh family, face and cell polynomial decompositions, and $\vec{H}(\Curl)\cap\vec{H}(\Div_\eta)$-like hybrid spaces. Section~\ref{se:weber} is dedicated to the statements and proofs of the first and second hybrid Weber inequalities. The related discrete Maxwell compactness properties are, in turn, established in Section~\ref{se:maxcom}. Finally, Appendix~\ref{ap:nim} collects a connected technical result.  


\section{Preliminaries} \label{se:prelim}

\subsection{Topological framework}

Let $\Omega$ be a domain in $\Real^3$, i.e.~a bounded and connected (strongly) Lipschitz open set of $\Real^3$. We let $\Gamma\defi\partial\Omega$ denote the boundary of $\Omega$. By the Rademacher theorem, one can define almost everywhere on $\Gamma$ a unit vector field $\normal$ normal to $\Gamma$, which we further assume to point outward from $\Omega$. We recall that the Betti numbers $\beta_0$, $\beta_1$ and $\beta_2$ of $\Omega$ respectively denote the number of (maximally) connected components of $\Omega$ (here, $\beta_0=1$), the number of tunnels crossing through $\Omega$ ($\beta_1\in\Natu$), and the number of voids enclosed in $\Omega$ ($\beta_2\in\Natu$). For instance, for $\Omega$ simply-connected, $\beta_1=0$. In the same manner, when the boundary $\Gamma$ of $\Omega$ is connected, $\beta_2=0$. In the following, when both $\beta_1$ and $\beta_2$ are equal to zero, we say that $\Omega$ has {\em trivial topology}.

When $\beta_1>0$, we make the following classical assumption: there exist $\beta_1$ non-intersecting, orientable, two-dimensional manifolds $\Sigma_1,\ldots,\Sigma_{\beta_1}$ with boundary, called {\em cutting surfaces}, satisfying $\partial\Sigma_i\subset\Gamma$ for all $i\in\{1,\ldots,\beta_1\}$, such that the open set $\hat{\Omega}\defi\Omega\setminus\cup_{i\in\{1,\ldots,\beta_1\}}\Sigma_i$ is not crossed by any tunnel (its first Betti number is then zero). Remark that $\partial\hat{\Omega}=\Gamma\cup\cup_{i\in\{1,\ldots,\beta_1\}}\Sigma_i$. We will assume in what follows that $\hat{\Omega}$ is connected (which is generally the case starting from a connected domain $\Omega$), and that the cutting surfaces are sufficiently regular so that the set $\hat{\Omega}$ is pseudo-Lipschitz (cf.~e.g.~\cite[Def.~3.2.2]{ACJLa:18}). Since the cutting surfaces $\Sigma_i$ are orientable, one can then define almost everywhere on $\Sigma_i$, for any $i\in\{1,\ldots,\beta_1\}$, a unit vector field $\normal_{\Sigma_i}$ normal to $\Sigma_i$, whose orientation is arbitrary but prescribed once and for all.

When $\beta_2>0$, letting $\Gamma_0$ be the (connected) boundary of the only unbounded component of the exterior open set $\Real^3\setminus\overline{\Omega}$, there exist $\beta_2$ (maximally) connected components $\Gamma_1,\ldots,\Gamma_{\beta_2}$ of $\Gamma$ such that $\Gamma=\cup_{j\in\{0,\ldots,\beta_2\}}\Gamma_j$. When $\beta_2=0$, there holds $\Gamma=\Gamma_0$.

\subsection{Functional setting} \label{sse:fun.set}

Let $\eta:\Omega\to\Real$ be a given function satisfying, for real numbers $0<\eta_\flat\leq\eta_\sharp<\infty$,
\begin{equation} \label{eq:eta.c}
  \eta_\flat\leq\eta(\vec{x})\leq\eta_\sharp\qquad\text{for a.e.}~\vec{x}\in\Omega.
\end{equation}
We also introduce $\kappa_\eta\defi\eta_\sharp/\eta_\flat\geq 1$ the (global) heterogeneity ratio of the parameter $\eta$. 

For $m\in\{2,3\}$, and for $X$ an $m$-dimensional, (relatively) open Lipschitz subset of $\overline{\Omega}$, we let $\LL[X]$ (respectively, $\bLL[X]$) denote the Lebesgue space of $\Comp$-valued functions (respectively, $\Comp^m$-valued vector fields) with square-integrable modulus on $X$. The standard Hermitian inner products (and norms) in $\LL[X]$ and $\bLL[X]$ are irrespectively denoted by $(\mathfrak{f},\mathfrak{g})_X\defi\int_X\mathfrak{f}{\cdot}\overline{\mathfrak{g}}$ (and $\|{\cdot}\|_{0,X}\defi\sqrt{({\cdot},{\cdot})_X}$, with the convention that $\|{\cdot}\|_{0}\defi\|{\cdot}\|_{0,\Omega}$). We also define
$$L^2_0(\Omega)\defi\left\{v\in\LL\mid\int_\Omega v=0\right\}\qquad\text{and}\qquad\vec{L}^2_{\vec{0}}(\Omega)\defi\left\{\vec{v}\in\bLL\mid\int_\Omega\vec{v}=\vec{0}\right\},$$
as well as the space $\vec{L}^2_\eta(\Omega)\defi\left(\bLL,(\eta\,\cdot,{\cdot})_{\Omega}\right)$ (remark that the norm $\|\eta^{\frac12}{\cdot}\|_0$ is equivalent, by~\eqref{eq:eta.c}, to the norm $\|{\cdot}\|_0$ in $\bLL$). For $s\in\Natu^\star$, we let $H^s(X)$ (respectively, $\vec{H}^s(X)$) denote the Sobolev space of $\Comp$-valued functions in $\LL[X]$ (respectively, $\Comp^m$-valued vector fields in $\bLL[X]$) whose partial weak derivatives of order up to $s$ have square-integrable modulus on $X$. The standard norms (and semi-norms) in $H^s(X)$ and $\vec{H}^s(X)$ are irrespectively denoted by $\|{\cdot}\|_{s,X}$ (and $|{\cdot}|_{s,X}$), with the convention that $\|{\cdot}\|_s\defi\|{\cdot}\|_{s,\Omega}$ (and $|{\cdot}|_s\defi|{\cdot}|_{s,\Omega}$). We also define $H^1_0(\Omega)\defi\{v\in H^1(\Omega)\mid v_{\mid\Gamma}=0\}$.

Let $Y$ be a three-dimensional, open Lipschitz subset of $\Omega$. Classically, we let
\begin{align*}
  \Hcurl[Y]&\defi\big\{\vec{v}\in\bLL[Y]\mid\Curl\vec{v}\in\bLL[Y]\big\},\\
  \Hdiv[Y]&\defi\big\{\vec{v}\in\bLL[Y]\mid\Div(\eta\vec{v})\in\LL[Y]\big\},
\end{align*}
as well as their subspaces
\begin{align*}
  \Hcurlz[Y]&\defi\big\{\vec{v}\in\Hcurl[Y]\mid\Curl\vec{v}\equiv\vec{0}\big\},\\
  \Hdivz[Y]&\defi\big\{\vec{v}\in\Hdiv[Y]\mid\Div(\eta\vec{v})\equiv 0\big\}.
\end{align*}
The 
spaces $\Hcurl[Y]$ and $\Hdiv[Y]$ are endowed with the following weighted norms:
$$\|{\cdot}\|_{\Curl,Y}^2\defi\|\eta^{\frac12}{\cdot}\|_{0,Y}^2+\eta_\sharp\|{\Curl\cdot}\|_{0,Y}^2,\qquad\|{\cdot}\|_{\Div,Y}^2\defi\|\eta^{\frac12}{\cdot}\|_{0,Y}^2+\eta_\flat^{-1}\|{\Div(\eta\,\cdot)}\|_{0,Y}^2,$$
with the convention that $\|{\cdot}\|_{\Curl}\defi\|{\cdot}\|_{\Curl,\Omega}$ and $\|{\cdot}\|_{\Div}\defi\|{\cdot}\|_{\Div,\Omega}$.
Let $\normal_{\partial Y}$ denote the unit outward normal vector field to $\partial Y$, defined almost everywhere on $\partial Y$.
For $\vec{v}\in\Hdiv[Y]$, one can give a sense to the normal trace of $\eta\vec{v}$ on $\partial Y$, denoted $(\eta\vec{v})_{\mid\partial Y}{\cdot}\normal_{\partial Y}$, as an element of $H^{-\frac12}(\partial Y)$ (space of bounded antilinear forms on $H^{\frac12}(\partial Y)$). In addition, the mapping $\vec{v}\mapsto(\eta\vec{v})_{\mid\partial Y}{\cdot}\normal_{\partial Y}$ is continuous from $\Hdiv[Y]$ to $H^{-\frac12}(\partial Y)$.
Likewise, for $\vec{v}\in\Hcurl[Y]$, one can give a sense to the rotated tangential trace of $\vec{v}$ on $\partial Y$, denoted $\vec{v}_{\mid\partial Y}{\times}\normal_{\partial Y}$, as an element of $H^{-\frac12}(\partial Y)^3$ (space of bounded antilinear forms on $H^{\frac12}(\partial Y)^3$), and the mapping $\vec{v}\mapsto\vec{v}_{\mid\partial Y}{\times}\normal_{\partial Y}$ is continuous from $\Hcurl[Y]$ to $H^{-\frac12}(\partial Y)^3$.
\begin{remark}[Tangential vector fields] \label{re:tvf}
  Remark that, with our above choice of notation, $\bLL[\partial Y]=\LL[\partial Y]^2$. Define $\vec{H}^{\frac12}(\partial Y)\defi H^{\frac12}(\partial Y)^2$ and $\vec{H}^{-\frac12}(\partial Y)\defi H^{-\frac12}(\partial Y)^2$. Letting $H_{\vec{x}}$ denote the tangent hyperplane to $\partial Y$ at point $\vec{x}$, and considering $\vec{\tau}_1(\vec{x}),\vec{\tau}_2(\vec{x})\subset H_{\vec{x}}$ two (orthonormal) vectors such that $\big(\vec{\tau}_1(\vec{x}),\vec{\tau}_2(\vec{x}),\normal_{\partial Y}(\vec{x})\big)$ forms a right-handed orthonormal basis of $\Real^3$, one can identify $(\vec{v}_{\mid\partial Y}{\times}\normal_{\partial Y})(\vec{x})$ with the vector of $\Comp^2$ of its local coordinates in the basis $\big(\vec{\tau}_1(\vec{x}),\vec{\tau}_2(\vec{x})\big)$ of $H_{\vec{x}}$.
With a slight abuse of notation, one can then write that $\vec{v}_{\mid\partial Y}{\times}\normal_{\partial Y}\in\vec{H}^{-\frac12}(\partial Y)$. We will extensively make use of this observation in the sequel.
\end{remark}
\noindent
In what follows, the (sesquilinear) duality pairings between $H^{-\frac12}(\partial Y)$ and $H^{\frac12}(\partial Y)$ on the one side, and between $\vec{H}^{-\frac12}(\partial Y)$ and $\vec{H}^{\frac12}(\partial Y)$ on the other side, are irrespectively denoted by $\langle{\cdot},{\cdot}\rangle_{\partial Y}$.
Whenever $\vec{v}\in\vec{H}^1(Y)\subset\Hcurl[Y]\cap\vec{H}(\Div;Y)$ with $\vec{H}(\Div;Y)\defi\vec{H}(\Div_1;Y)$, for almost every $\vec{x}\in\partial Y$,
$$\vec{v}_{\mid\partial Y}(\vec{x})=(\vec{v}_{\mid\partial Y}{\cdot}\normal_{\partial Y})(\vec{x})\normal_{\partial Y}(\vec{x})+\normal_{\partial Y}(\vec{x}){\times}(\vec{v}_{\mid\partial Y}{\times}\normal_{\partial Y})(\vec{x}).$$
In this case, $\vec{v}_{\mid\partial Y}{\cdot}\normal_{\partial Y}\in\LL[\partial Y]$ and $\vec{v}_{\mid\partial Y}{\times}\normal_{\partial Y}\in\bLL[\partial Y]$.

We finally introduce the following subspaces of $\Hcurl$ and $\Hdiv$:
$$\Hzcurl\defi\left\{\vec{v}\in\Hcurl\mid\vec{v}_{\mid\Gamma}{\times}\normal\equiv\vec{0}\right\}$$
and $\Hzcurlz\defi\Hzcurl\cap\Hcurlz$, as well as
$$\Hzdiv\defi\left\{\vec{v}\in\Hdiv\mid(\eta\vec{v})_{\mid\Gamma}{\cdot}\normal\equiv 0\right\}$$
and $\Hzdivz\defi\Hzdiv\cap\Hdivz$.

Assume that the first Betti number $\beta_1$ of $\Omega$ is positive. Then, for $v\in\LL[\hat{\Omega}]$ (resp.~$\vec{v}\in\bLL[\hat{\Omega}]$), we denote by $\check{v}$ (resp.~$\check{\vec{v}}$) its continuation to $\LL$ (resp.~$\bLL$). Also, for each cutting surface $\Sigma_i$ of $\Omega$, $i\in\{1,\ldots,\beta_1\}$, we associate the superscript $+$ to the side of $\hat{\Omega}$ (with respect to $\Sigma_i$) for which $\normal_{\Sigma_i}$ is outward, and the superscript $-$ to the side of $\hat{\Omega}$ for which $\normal_{\Sigma_i}$ is inward. Let $v$ be some function defined on $\hat{\Omega}$ and, with obvious notation, let $v^+_{\mid\Sigma_i}$ and $v^-_{\mid\Sigma_i}$ denote its two traces on $\Sigma_i$ defined (if need be, in a weak sense) from both sides of $\hat{\Omega}$. Then, the jump of $v$ on $\Sigma_i$ is defined by
\begin{equation} \label{eq:jump}
  \llbracket v\rrbracket_{\Sigma_i}\defi v^+_{\mid\Sigma_i}-v^-_{\mid\Sigma_i}.
\end{equation}                         

\subsection{Helmholtz--Hodge decompositions}

We introduce in this section two Helmholtz--Hodge decompositions for vector fields in $\bLL$.

\subsubsection{First Helmholtz--Hodge decomposition}

Let us introduce the following space of harmonic vector fields:
\begin{equation} \label{eq:fharmonic}
  \fharmonic\defi\Hzcurlz\cap\Hdivz.
\end{equation}
By~\cite[Prop.~3.3.10]{ACJLa:18}, the harmonic space $\fharmonic$ has dimension $\beta_2$, and a basis for $\fharmonic$ is given by $\big(\Grad\omega_j\big)_{j\in\{1,\ldots,\beta_2\}}$, where, for $j\in\{1,\ldots,\beta_2\}$, $\omega_j \in H^1(\Omega)$ is the unique solution to
\begin{subequations}
\label{eq:harmonic1}
  \begin{alignat}{2}
    -\Div (\eta\Grad \omega_j) &= 0 &\qquad&\text{in $\Omega$},
	\\
    \omega_j  &=  0 &\qquad&\text{on $\Gamma\setminus\Gamma_j$},
    \\
    \omega_j  &=  1 &\qquad&\text{on $\Gamma_j$}.
  \end{alignat}
\end{subequations}
Remark that the $\omega_j$'s are real-valued (since $\overline{\omega_j}$ also solves the uniquely solvable Problem~\eqref{eq:harmonic1}), and
$$\mathfrak{C}_{j'j}\defi\big(\eta\Grad\omega_j,\Grad\omega_{j'}\big)_{\Omega}=\big\langle(\eta\Grad\omega_j)_{\mid\Gamma_{j'}}{\cdot}\normal,1\big\rangle_{\Gamma_{j'}},$$
which implies that $\Grad\omega_j$ and $\Grad\omega_{j'}$ are not $\vec{L}^2_\eta(\Omega)$-orthogonal in general. The real-valued matrix $\mathfrak{C}\in\Real^{\beta_2\times\beta_2}$ is called {\em capacitance matrix}, and is symmetric positive-definite (cf.~\cite[Cor.~3.3.8]{ACJLa:18}). For $\vec{w}\in\fharmonic$ writing $\vec{w}=\Grad\omega$, with $\omega\in H^1(\Omega)$ such that $\omega\defi\sum_{j=1}^{\beta_2}\alpha_{j}\omega_{j}$, defining the vectors $\vec{\alpha}_{\vec{w}}\defi(\alpha_{j}\in\Comp)_{j\in\{1,\ldots,\beta_2\}}$ and $\vec{\beta}_{\vec{w}}\defi(\langle(\eta\vec{w})_{\mid\Gamma_{j'}}{\cdot}\normal,1\rangle_{\Gamma_{j'}}\in\Comp)_{j'\in\{1,\ldots,\beta_2\}}$, there holds
\begin{equation} \label{eq:capac}
  \vec{\alpha}_{\vec{w}}=\mathfrak{C}^{-1}\vec{\beta}_{\vec{w}}.
\end{equation}
As a consequence, $\vec{w}\equiv\vec{0}$ if and only if $\langle(\eta\vec{w})_{\mid\Gamma_{j'}}{\cdot}\normal,1\rangle_{\Gamma_{j'}}=0$ for all $j'\in\{1,\ldots,\beta_2\}$ (in fact, for all $j'\in\{0,\ldots,\beta_2\}$ since $\vec{w}\in\Hdivz$).

Building on~\cite[Prop.~3.7.1]{ACJLa:18}, and using~\cite[Thm.~3.4.1]{ACJLa:18}, we infer the following first $\vec{L}^2_\eta(\Omega)$-orthogonal Helmholtz--Hodge decomposition:
\begin{equation} \label{eq:helm1}
  \bLL=\Grad\big(H^1_0(\Omega)\big)\overset{\perp_\eta}{\oplus}\frac{\eta_\sharp}{\eta}\Curl\big(\vec{H}^1(\Omega)\cap\vec{L}^2_{\vec{0}}(\Omega)\big)\overset{\perp_\eta}{\oplus}\fharmonic.
\end{equation}
Furthermore, by~\cite[Thm.~3.4.1]{ACJLa:18}, we know that there is $C_{\Omega,1}>0$ such that $\vec{\psi}\in\vec{H}^1(\Omega)\cap\vec{L}^2_{\vec{0}}(\Omega)$ satisfying $\Curl\vec{\psi}=\vec{z}$ can always be chosen so that
\begin{equation} \label{hh.reg1}
  \|\vec{\psi}\|_1\leq C_{\Omega,1}\|\vec{z}\|_0.
\end{equation}

\subsubsection{Second Helmholtz--Hodge decomposition}

Let us now introduce the following space of harmonic vector fields:
\begin{equation} \label{eq:sharmonic}
  \sharmonic\defi\Hcurlz\cap\Hzdivz.
\end{equation}
By~\cite[Prop.~3.3.13]{ACJLa:18}, the harmonic space $\sharmonic$ has dimension $\beta_1$, and a basis for $\sharmonic$ is given by $\big(\check{\Grad\pi_i}\big)_{i\in\{1,\ldots,\beta_1\}}$, where, for $i\in\{1,\ldots,\beta_1\}$, $\pi_i \in H^1(\hat{\Omega})\cap L^2_0(\hat{\Omega})$ is the unique solution (recall that $\hat{\Omega}$ is assumed to be connected) to
\begin{subequations}
\label{eq:harmonic2}
  \begin{alignat}{2}
    -\Div (\eta\Grad \pi_i) &= 0 &\qquad&\text{in $\hat{\Omega}$},
    \\
    (\eta\Grad \pi_i) {\cdot} \normal &= 0 &\qquad&\text{on $\Gamma$},
	\\
    \llbracket(\eta\Grad \pi_i) {\cdot} \normal_{\Sigma_{i'}}\rrbracket_{\Sigma_{i'}}  &=  0 &\qquad&\text{for all $i'\in\{1,\ldots,\beta_1\}$},
    \\
    \llbracket\pi_i\rrbracket_{\Sigma_{i'}} &= \delta_{ii'}&\qquad&\text{for all $i'\in\{1,\ldots,\beta_1\}$}.
  \end{alignat}
\end{subequations}
We remind the reader that $\check{\Grad\pi_i}$ is the continuation to $\bLL$ of $\Grad\pi_i\in\bLL[\hat{\Omega}]$, and that the jump operator $\llbracket{\cdot}\rrbracket_{\Sigma_{i'}}$ is defined in~\eqref{eq:jump}. Notice that the $\pi_i$'s are real-valued (since $\overline{\pi_i}$ also solves the uniquely solvable Problem~\eqref{eq:harmonic2}) and, following the convention~\eqref{eq:jump},
$$\mathfrak{L}_{i'i}\defi\big(\eta\,\check{\Grad\pi_i},\check{\Grad\pi_{i'}}\big)_{\Omega}=\big(\eta\Grad\pi_i,\Grad\pi_{i'}\big)_{\hat{\Omega}}=\big\langle(\eta\Grad\pi_i)_{\mid\Sigma_{i'}}{\cdot}\normal_{\Sigma_{i'}},1\big\rangle_{\Sigma_{i'}},$$
which implies that $\check{\Grad\pi_i}$ and $\check{\Grad\pi_{i'}}$ are not $\vec{L}^2_\eta(\Omega)$-orthogonal in general.
The real-valued matrix $\mathfrak{L}\in\Real^{\beta_1\times\beta_1}$ is called {\em inductance matrix}, and is symmetric positive-definite (see~\cite[Cor.~3.3.14]{ACJLa:18}). For $\vec{w}\in\sharmonic$ writing $\vec{w}=\check{\Grad\pi}$, with $\pi\in H^1(\hat{\Omega})\cap L^2_0(\hat{\Omega})$ such that $\pi\defi\sum_{i=1}^{\beta_1}\alpha_{i}\pi_{i}$, letting this time $\vec{\alpha}_{\vec{w}}\defi(\alpha_{i}\in\Comp)_{i\in\{1,\ldots,\beta_1\}}$ and $\vec{\beta}_{\vec{w}}\defi(\langle(\eta\vec{w})_{\mid\Sigma_{i'}}{\cdot}\normal_{\Sigma_{i'}},1\rangle_{\Sigma_{i'}}\in\Comp)_{i'\in\{1,\ldots,\beta_1\}}$, there holds
\begin{equation} \label{eq:induct}
  \vec{\alpha}_{\vec{w}}=\mathfrak{L}^{-1}\vec{\beta}_{\vec{w}}.
\end{equation}
As a consequence, $\vec{w}\equiv\vec{0}$ if and only if $\langle(\eta\vec{w})_{\mid\Sigma_{i'}}{\cdot}\normal_{\Sigma_{i'}},1\rangle_{\Sigma_{i'}}=0$ for all $i'\in\{1,\ldots,\beta_1\}$.

Building on~\cite[Prop.~3.7.3]{ACJLa:18}, and using~\cite[Thm.~3.5.1]{ACJLa:18}, we infer the following second $\vec{L}^2_\eta(\Omega)$-orthogonal Helmholtz--Hodge decomposition:
\begin{equation} \label{eq:helm2}
  \bLL=\Grad\big(H^1(\Omega)\cap L^2_0(\Omega)\big)\overset{\perp_\eta}{\oplus}\frac{\eta_\sharp}{\eta}\Curl\big(\vec{H}^1(\Omega)\cap\Hzcurl\big)\overset{\perp_\eta}{\oplus}\sharmonic.
\end{equation}
Furthermore, combining the results of~\cite[Thm.~3.5.1]{ACJLa:18} and~\cite[Thm.~3.6.7]{ACJLa:18}, we know that there is $C_{\Omega,2}>0$ such that $\vec{\psi}\in\vec{H}^1(\Omega)\cap\Hzcurl$ satisfying $\Curl\vec{\psi}=\vec{z}$ can always be chosen so that
\begin{equation} \label{hh.reg2}
  \|\vec{\psi}\|_1\leq C_{\Omega,2}\|\vec{z}\|_0.
\end{equation}

\subsection{Weber inequalities and Maxwell compactness}

Recall the definition $\cweber\defi\Hcurl\cap\Hdiv$. Endowed with the (weighted) norm
$$\|\vec{v}\|^2_{\vec{X}}\defi\|\vec{v}\|_{\Div}^2+\|\vec{v}\|_{\Curl}^2,$$
$\cweber$ is a Hilbert space.
The first and second Weber inequalities, as well as the related Maxwell compactness properties, are direct consequences of the Helmholtz--Hodge decompositions~\eqref{eq:helm1} and~\eqref{eq:helm2}.

\subsubsection{First Weber inequality}

Let
\begin{equation} \label{eq:cfweb}
  \cfweber\defi\Hzcurl\cap\Hdiv.
\end{equation}
As a closed subspace of $\cweber$ (recall that the rotated tangential trace map is continuous), $\cfweber$, endowed with the $\|{\cdot}\|_{\vec{X}}$-norm, is a Hilbert space. We state and prove the first Weber inequality.
\begin{proposition}[First Weber inequality]
  There is $C_{W,1}>0$ only depending on $\Omega$ such that, for all $\vec{v}\in\cfweber$,
  \begin{equation} \label{eq:web1}
    \|\eta^{\frac12}\vec{v}\|_0\leq C_{W,1}\Bigg(\eta_\flat^{-\frac12}\|\Div(\eta\vec{v})\|_0+\eta_\sharp^{\frac12}\|\Curl\vec{v}\|_0+\eta_\flat^{-\frac12}\kappa_\eta^{\frac12}\left(\sum_{j=1}^{\beta_2}|\langle(\eta\vec{v})_{\mid\Gamma_j}{\cdot}\normal,1\rangle_{\Gamma_j}|^2\right)^{\nicefrac12}\Bigg).
  \end{equation}
\end{proposition}
\begin{proof}
  The proof of~\eqref{eq:web1} builds on the first Helmholtz--Hodge decomposition~\eqref{eq:helm1}. Since $\vec{v}\in\cfweber\subset\bLL$, there exist $\varphi\in H^1_0(\Omega)$, $\vec{\psi}\in\vec{H}^1(\Omega)\cap\vec{L}^2_{\vec{0}}(\Omega)$, and $\vec{w}\in\fharmonic$ (writing $\vec{w}=\Grad\omega$, with $\omega\in H^1(\Omega)$ such that $\omega\defi\sum_{j=1}^{\beta_2}\alpha_j\omega_j$), such that
  \begin{equation} \label{eq:decomp1}
    \vec{v}=\Grad\varphi+\frac{\eta_\sharp}{\eta}\Curl\vec{\psi}+\vec{w}.
  \end{equation}
  From the previous decomposition, there holds
  $$\|\eta^{\frac12}\vec{v}\|_0^2=(\eta\vec{v},\Grad\varphi)_{\Omega}+\eta_\sharp(\vec{v},\Curl\vec{\psi})_{\Omega}+(\eta\vec{v},\Grad\omega)_{\Omega}.$$
  By integration by parts, since $\varphi\in H^1_0(\Omega)$, $\vec{v}\in\Hzcurl$, and each $\omega_j$ satisfies~\eqref{eq:harmonic1}, we infer that
  $$\|\eta^{\frac12}\vec{v}\|_0^2=-(\Div(\eta\vec{v}),\varphi)_{\Omega}+\eta_\sharp(\Curl\vec{v},\vec{\psi})_{\Omega}-(\Div(\eta\vec{v}),\omega)_{\Omega}+\sum_{j=1}^{\beta_2}\langle(\eta\vec{v})_{\mid\Gamma_j}{\cdot}\normal,1\rangle_{\Gamma_j}\overline{\alpha_j}.$$
  The triangle and Cauchy--Schwarz inequalities, followed by the Poincar\'e inequality applied to both $\varphi\in H^1_0(\Omega)$ and $\omega\in H^1(\Omega)$ (remark that $\omega_{\mid\Gamma_0}=0$), and the estimate~\eqref{hh.reg1}, then yield
  \begin{multline*}
    \|\eta^{\frac12}\vec{v}\|_0^2\leq C\Bigg(\|\Div(\eta\vec{v})\|_0\|\Grad\varphi\|_0+\eta_\sharp^{\frac12}\|\Curl\vec{v}\|_0\,\eta_\sharp^{\frac12}\|\Curl\vec{\psi}\|_0\\+\|\Div(\eta\vec{v})\|_0\|\Grad\omega\|_0+\left(\sum_{j=1}^{\beta_2}|\langle(\eta\vec{v})_{\mid\Gamma_j}{\cdot}\normal,1\rangle_{\Gamma_j}|^2\right)^{\nicefrac12}\!\!|\vec{\alpha}_{\vec{w}}|\Bigg),
  \end{multline*}
  where $|\vec{\alpha}_{\vec{w}}|$ denotes the Euclidean norm of $\vec{\alpha}_{\vec{w}}\defi(\alpha_j\in\Comp)_{j\in\{1,\dots,\beta_2\}}$. Starting from the expression~\eqref{eq:capac} of $\vec{\alpha}_{\vec{w}}$, and since $\langle(\eta\vec{w})_{\mid\Gamma_{j'}}{\cdot}\normal,1\rangle_{\Gamma_{j'}}=(\eta\Grad\omega,\Grad\omega_{j'})_{\Omega}$, we infer that
  \begin{equation} \label{eq:alpha1}
    |\vec{\alpha}_{\vec{w}}|=|\mathfrak{C}^{-1}\vec{\beta}_{\vec{w}}|\leq\rho(\mathfrak{C}^{-1})|\vec{\beta}_{\vec{w}}|\leq\rho(\mathfrak{C}^{-1})\left(\sum_{j'=1}^{\beta_2}\|\eta^{\frac12}\Grad\omega_{j'}\|_0^2\right)^{\nicefrac12}\|\eta^{\frac12}\Grad\omega\|_0,
  \end{equation}
  where $\rho(\mathfrak{C}^{-1})>0$ denotes the spectral radius of the symmetric positive-definite (real-valued) matrix $\mathfrak{C}^{-1}$, which is proportional to $\eta_\flat^{-1}$.
  The conclusion then follows from the $\vec{L}^2_\eta(\Omega)$-orthogonality of the decomposition~\eqref{eq:decomp1}, so that $\|\eta^{\frac12}\Grad\varphi\|_0\leq\|\eta^{\frac12}\vec{v}\|_0$, $\eta_\sharp\|\eta^{-\frac12}\Curl\vec{\psi}\|_0\leq\|\eta^{\frac12}\vec{v}\|_0$, and $\|\eta^{\frac12}\Grad\omega\|_0\leq\|\eta^{\frac12}\vec{v}\|_0$, and from the fact that $\|{\cdot}\|_0\leq\eta_\flat^{-\frac12}\|\eta^{\frac12}{\cdot}\|_0$ and $\eta_\sharp^{\frac12}\|{\cdot}\|_0\leq\eta_\sharp\|\eta^{-\frac12}{\cdot}\|_0$.
\end{proof}
\noindent
The above proof of~\eqref{eq:web1} is constructive.
An alternative proof (in the case $\eta\equiv 1$), based on a contradiction argument, can be found in~\cite[Thm.~3.4.3]{ACJLa:18}.
Another proof (in the case of a symmetric tensor field $\eta$), leveraging the inverse mapping theorem, is available in~\cite[Thm.~6.1.6 ($s=0$)]{ACJLa:18}.
\begin{remark}[Useful variant] \label{rem:var1}
  Assume that $\vec{v}\in\Hzcurl$ satisfies
  $$(\eta\vec{v},\vec{z})_{\Omega}=0\qquad\forall\vec{z}\in\Grad\big(H^1_0(\Omega)\big)\overset{\perp_\eta}{\oplus}\fharmonic.$$
  Then, following the above proof, one can show that
  $$\|\eta^{\frac12}\vec{v}\|_0\leq C_{\Omega,1}\eta_\sharp^{\frac12}\|\Curl\vec{v}\|_0,$$
  where $C_{\Omega,1}>0$ is the multiplicative constant from~\eqref{hh.reg1}.
\end{remark}

The first Weber inequality enables us to define the following norm on $\cfweber$:
$$\|\vec{v}\|^2_{\vec{X}_{\vec{\tau}}}\defi\eta_\flat^{-1}\|\Div(\eta\vec{v})\|_0^2+\eta_\sharp\|\Curl\vec{v}\|_0^2+\eta_\flat^{-1}\sum_{j=1}^{\beta_2}|\langle(\eta\vec{v})_{\mid\Gamma_j}{\cdot}\normal,1\rangle_{\Gamma_j}|^2,$$
which is equivalent to the $\|{\cdot}\|_{\vec{X}}$-norm: there is $C_{\vec{\tau}}>0$, only depending on $\Omega$, such that
$$(C_{\vec{\tau}}\kappa_\eta)^{-1}\|\vec{v}\|_{\vec{X}}^2\leq\|\vec{v}\|_{\vec{X}_{\vec{\tau}}}^2\leq C_{\vec{\tau}}\kappa_\eta\|\vec{v}\|_{\vec{X}}^2\qquad\forall\vec{v}\in\cfweber.$$
Endowed with the $\|{\cdot}\|_{\vec{X}_{\vec{\tau}}}$-norm (more precisely, with the corresponding Hermitian inner product), $\cfweber$ is a Hilbert space. We are now in position to state and prove Maxwell compactness in $\cfweber$. Our proof partly takes inspiration from the one of~\cite[Thm.~3.4.4]{ACJLa:18} (see also~\cite[Thm.~7.5.1 ($s=0$)]{ACJLa:18} in the case of a symmetric tensor field $\eta$).

\begin{proposition}[Maxwell compactness in $\cfweber$] \label{pr:max1}
  Let $(\vec{v}_m)_{m\in\Natu}$ be a sequence of elements of $\cfweber$ for which there exists a real number $C_M>0$ such that $\eta_\sharp^{-\frac12}\|\vec{v}_m\|_{\vec{X}_{\vec{\tau}}}\leq C_M$ for all $m\in\Natu$. Then, there exists an element $\vec{v}\in\cfweber$ such that, along a subsequence as $m\to\infty$, $\vec{v}_m\to\vec{v}$ strongly in $\bLL$.
\end{proposition}
\begin{proof}
  The proof proceeds in three steps.

  \smallskip
  {\em Step 1 (Weak convergence):} By assumption, $\|\vec{v}_m\|_{\vec{X}_{\vec{\tau}}}\leq C_M\eta_\sharp^{\frac12}$ for all $m\in\Natu$. By definition of the $\|{\cdot}\|_{\vec{X}_{\vec{\tau}}}$-norm, and by the first Weber inequality~\eqref{eq:web1}, this implies that (i) $(\vec{v}_m)_{m\in\Natu}$ and $(\Curl\vec{v}_m)_{m\in\Natu}$ are uniformly (in $m$) bounded in $\bLL$, and (ii) $(\Div(\eta\vec{v}_m))_{m\in\Natu}$ is uniformly (in $m$) bounded in $\LL$. Using standard (weak compactness and limit regularity) arguments, we then infer the existence of $\vec{v}\in\cfweber$ such that, along a subsequence (not relabelled), (i) $\vec{v}_m\rightharpoonup\vec{v}$ and $\Curl\vec{v}_m\rightharpoonup\Curl\vec{v}$ weakly in $\bLL$, and (ii) $\Div(\eta\vec{v}_m)\rightharpoonup\Div(\eta\vec{v})$ weakly in $\LL$.

  \smallskip
  {\em Step 2 (Characterization of the limit):} For any $m\in\Natu$, since $\vec{v}_m\in\cfweber\subset\bLL$, invoking the first Helmholtz--Hodge decomposition~\eqref{eq:helm1}, there exist $\varphi_m\in H^1_0(\Omega)$, $\vec{\psi}_m\in\vec{H}^1(\Omega)\cap\vec{L}^2_{\vec{0}}(\Omega)$, and $\vec{w}_m\in\fharmonic$, such that
  \begin{equation} \label{eq:decomp1m}
    \vec{v}_m=\Grad\varphi_m+\frac{\eta_\sharp}{\eta}\Curl\vec{\psi}_m+\vec{w}_m,
  \end{equation}
  with $\|\vec{\psi}_m\|_1\leq C_{\Omega,1}\|\Curl\vec{\psi}_m\|_0$ by~\eqref{hh.reg1}.
  Since the sequence $(\vec{v}_m)_{m\in\Natu}$ is uniformly (in $m$) bounded in $\bLL$, and the decomposition~\eqref{eq:decomp1m} is $\vec{L}^2_\eta(\Omega)$-orthogonal with $\eta$ essentially bounded by above and by below away from zero in $\Omega$, the three sequences $(\Grad\varphi_m)_{m\in\Natu}$, $(\Curl\vec{\psi}_m)_{m\in\Natu}$, and $(\vec{w}_m)_{m\in\Natu}$ are also uniformly (in $m$) bounded in $\bLL$.
  Given that $(\vec{w}_m)_{m\in\Natu}\subset\fharmonic$, and that the harmonic space $\fharmonic$ has finite dimension, we directly infer from the Bolzano--Weierstra\ss~theorem the existence of $\vec{w}\in\fharmonic$ such that, up to extraction (not relabelled), $\vec{w}_m\to\vec{w}$ strongly in $\bLL$.
  Let us now deal with the two remaining terms of the decomposition.
  By the Poincar\'e inequality applied to $\varphi_m\in H^1_0(\Omega)$, and the estimate $\|\vec{\psi}_m\|_1\leq C_{\Omega,1}\|\Curl\vec{\psi}_m\|_0$, there holds that (i) $(\varphi_m)_{m\in\Natu}$ is uniformly bounded in $H^1(\Omega)$, and (ii) $(\vec{\psi}_m)_{m\in\Natu}$ is uniformly bounded in $\vec{H}^1(\Omega)$.
  Invoking standard (weak compactness and limit regularity) arguments, one can infer the existence of $\varphi\in H^1_0(\Omega)$ and $\vec{\psi}\in\vec{H}^1(\Omega)\cap\vec{L}^2_{\vec{0}}(\Omega)$ such that, up to extractions (not relabelled), $\Grad\varphi_m\rightharpoonup\Grad\varphi$ and $\Curl\vec{\psi}_m\rightharpoonup\Curl\vec{\psi}$ weakly in $\bLL$. By linearity, we have thus proven that, along a subsequence (not relabelled), $\vec{v}_m\rightharpoonup\Grad\varphi+\frac{\eta_\sharp}{\eta}\Curl\vec{\psi}+\vec{w}$ weakly in $\bLL$. The uniqueness of the weak limit then yields that $\vec{v}=\Grad\varphi+\frac{\eta_\sharp}{\eta}\Curl\vec{\psi}+\vec{w}$.

  \smallskip
  {\em Step 3 (Strong convergence):} By Rellich's compactness theorem, it actually holds that, along the same subsequence (not relabelled) as in Step 2, $\varphi_m\to\varphi$ strongly in $\LL$ and $\vec{\psi}_m\to\vec{\psi}$ strongly in $\bLL$.
   We now want to prove the strong convergences of $(\Grad\varphi_m)_{m\in\Natu}$ and $(\Curl\vec{\psi}_m)_{m\in\Natu}$ in $\bLL$ (we remind that we already proved strong convergence for $(\vec{w}_m)_{m\in\Natu}$). Recalling the expression of $\vec{v}\in\cfweber$ we derived in Step 2, since the decomposition~\eqref{eq:decomp1m} is $\vec{L}^2_\eta(\Omega)$-orthogonal, we have
  $$\|\eta^{\frac12}(\Grad\varphi_m-\Grad\varphi)\|_0^2=\big(\eta(\vec{v}_m-\vec{v}),\Grad(\varphi_m-\varphi)\big)_{\Omega}=-\big(\Div(\eta(\vec{v}_m-\vec{v})),\varphi_m-\varphi\big)_{\Omega},$$
  where we have used that $\varphi_m,\varphi\in H^1_0(\Omega)$. By Cauchy--Schwarz inequality, combined with the fact that $\|\Div(\eta\vec{v}_m)\|_0\leq C_M\eta_\sharp$ for all $m\in\Natu$ and $\Div(\eta\vec{v}_m)\rightharpoonup\Div(\eta\vec{v})$ weakly in $\LL$ (so that $\|\Div(\eta\vec{v})\|_0\leq C_M\eta_\sharp$), we infer that
  $$\|\eta^{\frac12}(\Grad\varphi_m-\Grad\varphi)\|_0^2\leq 2C_M\eta_\sharp\|\varphi_m-\varphi\|_0.$$
  Since $\varphi_m\to\varphi$ strongly in $\LL$, passing to the limit $m\to\infty$ we conclude that $\Grad\varphi_m\to\Grad\varphi$ strongly in $\bLL$. By the very same arguments, there holds
  $$\|\eta^{-\frac12}(\Curl\vec{\psi}_m-\Curl\vec{\psi})\|_0^2=\eta_\sharp^{-1}\big(\vec{v}_m-\vec{v},\Curl(\vec{\psi}_m-\vec{\psi})\big)_{\Omega}=\eta_\sharp^{-1}\big(\Curl(\vec{v}_m-\vec{v}),\vec{\psi}_m-\vec{\psi}\big)_{\Omega},$$
  where we have used that $\vec{v}_m,\vec{v}\in\Hzcurl$. This implies that
  $$\|\eta^{-\frac12}(\Curl\vec{\psi}_m-\Curl\vec{\psi})\|_0^2\leq 2C_M\eta_\sharp^{-1}\|\vec{\psi}_m-\vec{\psi}\|_0,$$
  which eventually yields that $\Curl\vec{\psi}_m\to\Curl\vec{\psi}$ strongly in $\bLL$. Thus, along a subsequence (not relabelled), $\vec{v}_m\to\vec{v}$ strongly in $\bLL$, which concludes the proof.
\end{proof}

\subsubsection{Second Weber inequality}

Let
\begin{equation} \label{eq:csweb}
  \csweber\defi\Hcurl\cap\Hzdiv.
\end{equation}
As a closed subspace of $\cweber$ (recall that the $\eta$-weighted normal trace map is continuous), $\csweber$, endowed with the $\|{\cdot}\|_{\vec{X}}$-norm, is a Hilbert space.
We now state and prove the second Weber inequality.
\begin{proposition}[Second Weber inequality]
  There is $C_{W,2}>0$ only depending on $\Omega$ such that, for all $\vec{v}\in\csweber$,
  \begin{equation} \label{eq:web2}
    \|\eta^{\frac12}\vec{v}\|_0\leq C_{W,2}\Bigg(\eta_\flat^{-\frac12}\|\Div(\eta\vec{v})\|_0+\eta_\sharp^{\frac12}\|\Curl\vec{v}\|_0+\eta_\flat^{-\frac12}\kappa_\eta^{\frac12}\left(\sum_{i=1}^{\beta_1}|\langle(\eta\vec{v})_{\mid\Sigma_i}{\cdot}\normal_{\Sigma_i},1\rangle_{\Sigma_i}|^2\right)^{\nicefrac12}\Bigg).
  \end{equation}
\end{proposition}
\begin{proof}
  The proof of~\eqref{eq:web2} builds on the second Helmholtz--Hodge decomposition~\eqref{eq:helm2}. Since $\vec{v}\in\csweber\subset\bLL$, there exist $\varphi\in H^1(\Omega)\cap L^2_0(\Omega)$, $\vec{\psi}\in\vec{H}^1(\Omega)\cap\Hzcurl$, and $\vec{w}\in\sharmonic$ (writing $\vec{w}=\check{\Grad\pi}$, with $\pi\in H^1(\hat{\Omega})\cap L^2_0(\hat{\Omega})$ such that $\pi\defi\sum_{i=1}^{\beta_1}\alpha_i\pi_i$), such that
  \begin{equation} \label{eq:decomp2}
    \vec{v}=\Grad\varphi+\frac{\eta_\sharp}{\eta}\Curl\vec{\psi}+\vec{w}.
  \end{equation}
  From the previous decomposition, there holds
  $$\|\eta^{\frac12}\vec{v}\|_0^2=(\eta\vec{v},\Grad\varphi)_{\Omega}+\eta_\sharp(\vec{v},\Curl\vec{\psi})_{\Omega}+(\eta\vec{v},\Grad\pi)_{\hat{\Omega}}.$$
  By integration by parts, since $\vec{v}\in\Hzdiv$, $\vec{\psi}\in\Hzcurl$, and each $\pi_i$ satisfies~\eqref{eq:harmonic2}, we infer
  $$\|\eta^{\frac12}\vec{v}\|_0^2=-(\Div(\eta\vec{v}),\varphi)_{\Omega}+\eta_\sharp(\Curl\vec{v},\vec{\psi})_{\Omega}-(\Div(\eta\vec{v}),\pi)_{\hat{\Omega}}+\sum_{i=1}^{\beta_1}\langle(\eta\vec{v})_{\mid\Sigma_i}{\cdot}\normal_{\Sigma_i},1\rangle_{\Sigma_i}\overline{\alpha_i}.$$
  The triangle and Cauchy--Schwarz inequalities, followed by the Poincar\'e--Steklov inequality applied to both $\varphi\in H^1(\Omega)\cap L^2_0(\Omega)$ and $\pi\in H^1(\hat{\Omega})\cap L^2_0(\hat{\Omega})$ (recall that $\hat{\Omega}$ is assumed to be connected), and the estimate~\eqref{hh.reg2}, then yield
  \begin{multline*}
    \|\eta^{\frac12}\vec{v}\|_0^2\leq C\Bigg(\|\Div(\eta\vec{v})\|_0\|\Grad\varphi\|_0+\eta_\sharp^{\frac12}\|\Curl\vec{v}\|_0\,\eta_\sharp^{\frac12}\|\Curl\vec{\psi}\|_0\\+\|\Div(\eta\vec{v})\|_0\|\check{\Grad\pi}\|_0+\left(\sum_{i=1}^{\beta_1}|\langle(\eta\vec{v})_{\mid\Sigma_i}{\cdot}\normal_{\Sigma_i},1\rangle_{\Sigma_i}|^2\right)^{\nicefrac12}\!\!|\vec{\alpha}_{\vec{w}}|\Bigg),
  \end{multline*}
  where $|\vec{\alpha}_{\vec{w}}|$ denotes the Euclidean norm of $\vec{\alpha}_{\vec{w}}\defi(\alpha_i\in\Comp)_{i\in\{1,\dots,\beta_1\}}$, and where we also used that $\|\Div(\eta\vec{v})\|_{0,\hat{\Omega}}=\|\Div(\eta\vec{v})\|_0$ since $\vec{v}\in\Hdiv$, and that $\|\Grad\pi\|_{0,\hat{\Omega}}=\|\check{\Grad\pi}\|_0$.
  Starting from the expression~\eqref{eq:induct} of $\vec{\alpha}_{\vec{w}}$, and since $\langle(\eta\vec{w})_{\mid\Sigma_{i'}}{\cdot}\normal_{\Sigma_{i'}},1\rangle_{\Sigma_{i'}}=(\eta\,\check{\Grad\pi},\check{\Grad\pi_{i'}})_{\Omega}$, we infer that
  \begin{equation} \label{eq:alpha2}
    |\vec{\alpha}_{\vec{w}}|=|\mathfrak{L}^{-1}\vec{\beta}_{\vec{w}}|\leq\rho(\mathfrak{L}^{-1})|\vec{\beta}_{\vec{w}}|\leq\rho(\mathfrak{L}^{-1})\left(\sum_{i'=1}^{\beta_1}\|\eta^{\frac12}\check{\Grad\pi_{i'}}\|_0^2\right)^{\nicefrac12}\|\eta^{\frac12}\check{\Grad\pi}\|_0,
  \end{equation}
  where $\rho(\mathfrak{L}^{-1})>0$ denotes the spectral radius of the symmetric positive-definite (real-valued) matrix $\mathfrak{L}^{-1}$, which is proportional to $\eta_\flat^{-1}$.
  The conclusion then follows from the $\vec{L}^2_\eta(\Omega)$-orthogonality of the decomposition~\eqref{eq:decomp2}, so that $\|\eta^{\frac12}\Grad\varphi\|_0\leq\|\eta^{\frac12}\vec{v}\|_0$, $\eta_\sharp\|\eta^{-\frac12}\Curl\vec{\psi}\|_0\leq\|\eta^{\frac12}\vec{v}\|_0$, and $\|\eta^{\frac12}\check{\Grad\pi}\|_0\leq\|\eta^{\frac12}\vec{v}\|_0$, and from the fact that $\|{\cdot}\|_0\leq\eta_\flat^{-\frac12}\|\eta^{\frac12}{\cdot}\|_0$ and $\eta_\sharp^{\frac12}\|{\cdot}\|_0\leq\eta_\sharp\|\eta^{-\frac12}{\cdot}\|_0$.
\end{proof}
\noindent
As for~\eqref{eq:web1}, the above proof of~\eqref{eq:web2} is constructive. An alternative proof (in the case $\eta\equiv 1$), based on a contradiction argument, can be found in~\cite[Thm.~3.5.3]{ACJLa:18}. Another proof (in the case of a symmetric tensor field $\eta$), leveraging the inverse mapping theorem, is available in~\cite[Thm.~6.2.5]{ACJLa:18}.
\begin{remark}[Useful variant] \label{rem:var2}
  Assume that $\vec{v}\in\Hcurl$ satisfies
  $$(\eta\vec{v},\vec{z})_{\Omega}=0\qquad\forall\vec{z}\in\Grad\big(H^1(\Omega)\cap L^2_0(\Omega)\big)\overset{\perp_\eta}{\oplus}\sharmonic.$$
  Then, revisiting the above proof, one can show that
  $$\|\eta^{\frac12}\vec{v}\|_0\leq C_{\Omega,2}\eta_\sharp^{\frac12}\|\Curl\vec{v}\|_0,$$
  where $C_{\Omega,2}>0$ is the multiplicative constant from~\eqref{hh.reg2}.
\end{remark}

The second Weber inequality enables us to define the following norm on $\csweber$:
$$\|\vec{v}\|^2_{\vec{X}_{n}}\defi\eta_\flat^{-1}\|\Div(\eta\vec{v})\|_0^2+\eta_\sharp\|\Curl\vec{v}\|_0^2+\eta_\flat^{-1}\sum_{i=1}^{\beta_1}|\langle(\eta\vec{v})_{\mid\Sigma_i}{\cdot}\normal_{\Sigma_i},1\rangle_{\Sigma_i}|^2,$$
which is equivalent to the $\|{\cdot}\|_{\vec{X}}$-norm: there is $C_{n}>0$, only depending on $\Omega$, such that
$$(C_{n}\kappa_\eta)^{-1}\|\vec{v}\|_{\vec{X}}^2\leq\|\vec{v}\|_{\vec{X}_{n}}^2\leq C_{n}\kappa_\eta\|\vec{v}\|_{\vec{X}}^2\qquad\forall\vec{v}\in\csweber.$$
Endowed with the $\|{\cdot}\|_{\vec{X}_{n}}$-norm (more precisely, with the corresponding Hermitian inner product), $\csweber$ is a Hilbert space. We are now in position to state and prove Maxwell compactness in $\csweber$. Our proof partly takes inspiration from the one of~\cite[Thm.~3.5.4]{ACJLa:18} (see also~\cite[Thm.~7.5.3]{ACJLa:18} in the case of a symmetric tensor field $\eta$).

\begin{proposition}[Maxwell compactness in $\csweber$] \label{pr:max2}
  Let $(\vec{v}_m)_{m\in\Natu}$ be a sequence of elements of $\csweber$ for which there exists a real number $C_M>0$ such that $\eta_\sharp^{-\frac12}\|\vec{v}_m\|_{\vec{X}_{n}}\leq C_M$ for all $m\in\Natu$. Then, there exists an element $\vec{v}\in\csweber$ such that, along a subsequence as $m\to\infty$, $\vec{v}_m\to\vec{v}$ strongly in $\bLL$.
\end{proposition}
\begin{proof}
  The proof proceeds in three steps.

  \smallskip
  {\em Step 1 (Weak convergence):} By assumption, $\|\vec{v}_m\|_{\vec{X}_{n}}\leq C_M\eta_\sharp^{\frac12}$ for all $m\in\Natu$. By definition of the $\|{\cdot}\|_{\vec{X}_{n}}$-norm, and by the second Weber inequality~\eqref{eq:web2}, this implies that (i) $(\vec{v}_m)_{m\in\Natu}$ and $(\Curl\vec{v}_m)_{m\in\Natu}$ are uniformly (in $m$) bounded in $\bLL$, and (ii) $(\Div(\eta\vec{v}_m))_{m\in\Natu}$ is uniformly (in $m$) bounded in $\LL$. Using standard (weak compactness and limit regularity) arguments, we then infer the existence of $\vec{v}\in\csweber$ such that, along a subsequence (not relabelled), (i) $\vec{v}_m\rightharpoonup\vec{v}$ and $\Curl\vec{v}_m\rightharpoonup\Curl\vec{v}$ weakly in $\bLL$, and (ii) $\Div(\eta\vec{v}_m)\rightharpoonup\Div(\eta\vec{v})$ weakly in $\LL$.

  \smallskip
  {\em Step 2 (Characterization of the limit):} For any $m\in\Natu$, since $\vec{v}_m\in\csweber\subset\bLL$, invoking the second Helmholtz--Hodge decomposition~\eqref{eq:helm2}, there exist $\varphi_m\in H^1(\Omega)\cap L^2_0(\Omega)$, $\vec{\psi}_m\in\vec{H}^1(\Omega)\cap\Hzcurl$, and $\vec{w}_m\in\sharmonic$, such that
  \begin{equation} \label{eq:decomp2m}
    \vec{v}_m=\Grad\varphi_m+\frac{\eta_\sharp}{\eta}\Curl\vec{\psi}_m+\vec{w}_m,
  \end{equation}
  with $\|\vec{\psi}_m\|_1\leq C_{\Omega,2}\|\Curl\vec{\psi}_m\|_0$ by~\eqref{hh.reg2}.
  Since the sequence $(\vec{v}_m)_{m\in\Natu}$ is uniformly (in $m$) bounded in $\bLL$, and the decomposition~\eqref{eq:decomp2m} is $\vec{L}^2_\eta(\Omega)$-orthogonal with $\eta$ essentially bounded by above and by below away from zero in $\Omega$, the three sequences $(\Grad\varphi_m)_{m\in\Natu}$, $(\Curl\vec{\psi}_m)_{m\in\Natu}$, and $(\vec{w}_m)_{m\in\Natu}$ are also uniformly (in $m$) bounded in $\bLL$.
  Given that $(\vec{w}_m)_{m\in\Natu}\subset\sharmonic$, and that the harmonic space $\sharmonic$ has finite dimension, we straightforwardly infer from the Bolzano--Weierstra\ss~theorem the existence of $\vec{w}\in\sharmonic$ such that, up to extraction (not relabelled), $\vec{w}_m\to\vec{w}$ strongly in $\bLL$.
  Let us now deal with the two remaining terms of the decomposition. By the Poincar\'e--Steklov inequality applied to $\varphi_m\in H^1(\Omega)\cap L^2_0(\Omega)$, and the estimate $\|\vec{\psi}_m\|_1\leq C_{\Omega,2}\|\Curl\vec{\psi}_m\|_0$, there holds that (i) $(\varphi_m)_{m\in\Natu}$ is uniformly bounded in $H^1(\Omega)$, and (ii) $(\vec{\psi}_m)_{m\in\Natu}$ is uniformly bounded in $\vec{H}^1(\Omega)$.
  Invoking standard (weak compactness and limit regularity) arguments, one can infer the existence of $\varphi\in H^1(\Omega)\cap L^2_0(\Omega)$ and $\vec{\psi}\in\vec{H}^1(\Omega)\cap\Hzcurl$ such that, up to extractions (not relabelled), $\Grad\varphi_m\rightharpoonup\Grad\varphi$ and $\Curl\vec{\psi}_m\rightharpoonup\Curl\vec{\psi}$ weakly in $\bLL$. By linearity, we have thus proven that, along a subsequence (not relabelled), $\vec{v}_m\rightharpoonup\Grad\varphi+\frac{\eta_\sharp}{\eta}\Curl\vec{\psi}+\vec{w}$ weakly in $\bLL$. The uniqueness of the weak limit then yields that $\vec{v}=\Grad\varphi+\frac{\eta_\sharp}{\eta}\Curl\vec{\psi}+\vec{w}$.

  \smallskip
  {\em Step 3 (Strong convergence):} By Rellich's compactness theorem, it actually holds that, along the same subsequence (not relabelled) as in Step 2, $\varphi_m\to\varphi$ strongly in $\LL$ and $\vec{\psi}_m\to\vec{\psi}$ strongly in $\bLL$.
  We now want to prove the strong convergences of $(\Grad\varphi_m)_{m\in\Natu}$ and $(\Curl\vec{\psi}_m)_{m\in\Natu}$ in $\bLL$ (we remind that we already proved strong convergence for $(\vec{w}_m)_{m\in\Natu}$). Recalling the expression of $\vec{v}\in\csweber$ we derived in Step 2, since the decomposition~\eqref{eq:decomp2m} is $\vec{L}^2_\eta(\Omega)$-orthogonal, we have
  $$\|\eta^{\frac12}(\Grad\varphi_m-\Grad\varphi)\|_0^2=\big(\eta(\vec{v}_m-\vec{v}),\Grad(\varphi_m-\varphi)\big)_{\Omega}=-\big(\Div(\eta(\vec{v}_m-\vec{v})),\varphi_m-\varphi\big)_{\Omega},$$
  where we have used that $\vec{v}_m,\vec{v}\in\Hzdiv$. By Cauchy--Schwarz inequality, combined with the fact that $\|\Div(\eta\vec{v}_m)\|_0\leq C_M\eta_\sharp$ for all $m\in\Natu$ and $\Div(\eta\vec{v}_m)\rightharpoonup\Div(\eta\vec{v})$ weakly in $\LL$ (so that $\|\Div(\eta\vec{v})\|_0\leq C_M\eta_\sharp$), we infer that
  $$\|\eta^{\frac12}(\Grad\varphi_m-\Grad\varphi)\|_0^2\leq 2C_M\eta_\sharp\|\varphi_m-\varphi\|_0.$$
  Since $\varphi_m\to\varphi$ strongly in $\LL$, passing to the limit $m\to\infty$ we conclude that $\Grad\varphi_m\to\Grad\varphi$ strongly in $\bLL$. By the very same arguments, there holds
  $$\|\eta^{-\frac12}(\Curl\vec{\psi}_m-\Curl\vec{\psi})\|_0^2=\eta_\sharp^{-1}\big(\vec{v}_m-\vec{v},\Curl(\vec{\psi}_m-\vec{\psi})\big)_{\Omega}=\eta_\sharp^{-1}\big(\Curl(\vec{v}_m-\vec{v}),\vec{\psi}_m-\vec{\psi}\big)_{\Omega},$$
  where we have used that $\vec{\psi}_m,\vec{\psi}\in\Hzcurl$. This implies that
  $$\|\eta^{-\frac12}(\Curl\vec{\psi}_m-\Curl\vec{\psi})\|_0^2\leq 2C_M\eta_\sharp^{-1}\|\vec{\psi}_m-\vec{\psi}\|_0,$$
  which eventually yields that $\Curl\vec{\psi}_m\to\Curl\vec{\psi}$ strongly in $\bLL$.
  Thus, along a subsequence (not relabelled), $\vec{v}_m\to\vec{v}$ strongly in $\bLL$, which concludes the proof.
\end{proof}

\section{Discrete setting} \label{se:dis.set}

From now on, we assume that the domain $\Omega\subset\Real^3$ is a (Lipschitz) polyhedron.

\subsection{Polyhedral meshes} \label{sse:pol.mesh}

We consider meshes $\mathcal{M}_h\defi(\t_h,\f_h)$ of $\Omega\subset\Real^3$ in the sense of~\cite[Def.~1.4]{DPDro:20}.

The set $\t_h$ is a finite collection of disjoint, open Lipschitz polyhedra $T$ (called {\em mesh cells}) such that $\overline{\Omega} = \bigcup_{T\in\t_h} \overline{T}$.
The subscript $h$ here refers to the mesh size, defined by $h \defi \max_{T\in\t_h} h_T$, where $h_T\defi\max_{\vec{x},\vec{y}\in\overline{T}}|\vec{x}-\vec{y}|$ denotes the diameter of the mesh cell $T$. The set $\f_h$ is a finite collection of disjoint, two-dimensional connected subsets of $\overline{\Omega}$ (called {\em mesh faces}) such that, for any $F\in\f_h$, (i) $F$ is a relatively open Lipschitz polygonal subset of an affine hyperplane, and (ii) either there exist two distinct mesh cells $T^+,T^-\in\t_h$ such that $\overline{F}\subseteq\partial T^+\cap\partial T^-$ ($F$ is then an {\em interface}), or there exists one mesh cell $T\in\t_h$ such that $\overline{F}\subseteq\partial T\cap\Gamma$ ($F$ is then a {\em boundary face}). The set of mesh faces is further assumed to satisfy $\bigcup_{T\in\t_h}\partial T=\bigcup_{F\in\f_h}\overline{F}$. 
For all $F\in\f_h$, we let $h_F\defi\max_{\vec{x},\vec{y}\in\overline{F}}|\vec{x}-\vec{y}|$ denote the diameter of the face $F$.
Interfaces are collected in the set $\f_h^\circ$, whereas boundary faces are collected in the set $\f_h^\partial$.
For all $T\in\t_h$, we denote by $\f_T$ the subset of $\f_h$ collecting those mesh faces lying on the boundary of $T$, so that $\partial T=\bigcup_{F\in\f_T}\overline{F}$.
For all $T\in\t_h$, consistently with our notation so far, we let $\normal_{\partial T}$ denote the unit vector field, defined almost everywhere on $\partial T$, normal to $\partial T$ and pointing outward from $T$. For all $F\in\f_T$, we also let $\normal_{T,F}\defi\normal_{\partial T\mid F}$ denote the (constant) unit vector normal to the hyperplane containing $F$ and pointing outward from $T$.
Finally, for all $F\in\f_h$, we define $\normal_F$ as the unit (constant) vector normal to $F$ such that either $\normal_F\defi\normal_{T^+,F}$ if $F\subset\partial T^+\cap\partial T^-\in\f_h^\circ$, or $\normal_F\defi\normal_{T,F}(=\normal_{\mid F})$ if $F\subset\partial T\cap\Gamma\in\f_h^\partial$. For further use, we also let, for all $T\in\t_h$ and all $F\in\f_T$, $\varepsilon_{T,F}\in\{-1,1\}$ be such that $\varepsilon_{T,F}\defi\normal_{T,F}{\cdot}\normal_F$.

When $\beta_1>0$, for all $i\in\{1,\ldots,\beta_1\}$, we assume that there exists a subset $\f^\circ_{h,\Sigma_i}$ of $\f_h^\circ$ such that $\overline{\Sigma_i}=\bigcup_{F\in\f^\circ_{h,\Sigma_i}}\overline{F}$, and for which $\normal_{F}=\normal_{\Sigma_i\mid F}$ for all $F\in\f^\circ_{h,\Sigma_i}$. This ensures that we do associate the superscript $+$ to the side of $\hat{\Omega}$ (with respect to $\Sigma_i$) for which $\normal_{\Sigma_i}$ is outward, consistently with our assumption from Section~\ref{sse:fun.set}. Remark that, since the cutting surfaces are piecewise planar, the set $\hat{\Omega}$ is indeed pseudo-Lipschitz. We also set
$$\hat{\f}_h^\circ\defi\f_h^\circ\setminus\bigcup_{i=1}^{\beta_1}\f_{h,\Sigma_i}^\circ.$$
When $\beta_2>0$, for all $j\in\{0,\ldots,\beta_2\}$, we let $\f^\partial_{h,\Gamma_j}$ denote the subset of $\f_h^\partial$ such that $\Gamma_j=\bigcup_{F\in\f^\partial_{h,\Gamma_j}}\overline{F}$.

At the discrete level, the (real-valued) parameter $\eta:\Omega\to[\eta_\flat,\eta_\sharp]$ introduced in~\eqref{eq:eta.c} is assumed to be piecewise constant on the partition $\t_h$ of the domain $\Omega$, and we let
\begin{equation} \label{eq:eta.d}
  \eta_\flat\leq\eta_T\defi\eta_{\mid T}\leq\eta_\sharp\qquad\forall T\in\t_h.
\end{equation}

Let $\Hindex\subset(0,\infty)$ be a countable set of mesh sizes having $0$ as its unique accumulation point. When studying asymptotic properties with respect to the mesh size, one has to adopt a measure of regularity for refined mesh families $(\mathcal{M}_h)_{h\in\Hindex}$. We classically follow~\cite[Def.~1.9]{DPDro:20}, in which regularity for refined mesh families is quantified through a uniform (with respect to $h$) parameter $\varrho\in(0,1)$, called {\em mesh regularity parameter}.
In a nutshell, it is assumed that, for all $h\in\Hindex$, there exists a matching tetrahedral submesh of $\mathcal{M}_h$, (i) which is uniformly shape-regular, and (ii) whose elements have a diameter that is uniformly comparable to the diameter of the cell in $\mathcal{M}_h$ they belong to.
In what follows, we write $a\lesssim b$ (resp.~$a\gtrsim b$) in place of $a\leq Cb$ (resp.~$a\geq Cb$), if $C>0$ only depends on $\Omega$, on the mesh regularity parameter $\varrho$, and (if need be) on the underlying polynomial degree, but is independent of both $h$ and $\eta$.
In particular, for regular mesh families $(\mathcal{M}_h)_{h\in\Hindex}$, for all $h\in\Hindex$, and for all $T\in\t_h$, there holds ${\rm card}(\f_T)\lesssim 1$, as well as $h_T\lesssim h_F\leq h_T$ for all $F\in\f_T$ (cf.~\cite[Lem.~1.12]{DPDro:20}).
Finally, for $X\in\t_h\cup\f_h$, we let $\vec{x}_X\in\Real^3$ be some point inside $X$ such that $X$ contains a ball/disk centered at $\vec{x}_X$ of radius $h_X\lesssim r_X\leq h_X$. The existence of such a point is guaranteed for regular mesh families.

\subsection{Polynomial spaces}

For $\ell\in\Natu$ and $m\in\{2,3\}$, we let $\mathcal{P}^\ell_m$ denote the linear space of $m$-variate $\Comp$-valued polynomials of total degree at most $\ell$, with the convention that, for any $m$, $\mathcal{P}^0_m$ is identified to $\Comp$, and $\mathcal{P}^{-1}_m\defi\{0\}$.
For any $X\in\t_h\cup\f_h$, we denote by $\Poly^\ell(X)$ the linear space spanned by the restrictions to $X$ of the polynomials in $\mathcal{P}^\ell_3$. Letting $m\in\{2,3\}$ be the dimension of $X$, $\Poly^\ell(X)$ is isomorphic to $\mathcal{P}^\ell_m$ (cf.~\cite[Prop.~1.23]{DPDro:20}). We let $\pi^{\ell}_{X}$ denote the $\LL[X]$-orthogonal projector onto $\Poly^{\ell}(X)$. For convenience, we also let $\bPoly^\ell(X)\defi\Poly^\ell(X)^m$ (i.e.~for all $T\in\t_h$, $\bPoly^\ell(T)=\Poly^\ell(T)^3$, and for all $F\in\f_h$, $\bPoly^\ell(F)=\Poly^\ell(F)^2$), and we define $\vec{\pi}^{\ell}_X$ as the $\bLL[X]$-orthogonal projector onto $\bPoly^{\ell}(X)$.

For any $F\in\f_h$, we let $H_F$ denote the hyperplane containing $F$, that we orient according to the normal $\normal_F$. For $w:F\to\Comp$, we let $\Grad_Fw:F\to\Comp^2$ denote the (tangential) gradient of $w$. Likewise, for $\vec{w}:F\to\Comp^2$, $\Div_F\vec{w}:F\to\Comp$ denotes the (tangential) divergence of $\vec{w}$. We also let $\mathbf{rot}_F w:F\to\Comp^2$ be such that
$$\mathbf{rot}_F w\defi\big(\Grad_Fw\big)^\perp,$$
where $\vec{z}^\perp$ is the rotation of angle $-\pi/2$ of $\vec{z}$ in the oriented hyperplane $H_F$.
Let us introduce the following subspaces of $\bPoly^{\ell}(F)$, $\ell\in\Natu$:
$$\pRF{\ell}\defi\mathbf{rot}_F\big(\Poly^{\ell+1}(F)\big),\qquad\kRF{\ell}\defi\Poly^{\ell-1}(F)(\vec{x}-\vec{x}_F),$$
where, for $\vec{x}\in F$, $(\vec{x}-\vec{x}_F)\subset H_F$ is identified to its two-dimensional tangential counterpart (cf.~Remark~\ref{re:tvf}). The polynomial space $\kRF{\ell}$ is the so-called {\em Koszul complement} of $\pRF{\ell}$.
As a consequence of the {\em homotopy formula} (cf.~\cite[Thm.~7.1]{Arnol:18}), the following (non $\bLL[F]$-orthogonal) polynomial decomposition holds true:
\begin{equation} \label{eq:decomp.F}
  \bPoly^\ell(F)=\pRF{\ell}\oplus\kRF{\ell}.
\end{equation}
Furthermore, by exactness of the (tangential) polynomial de Rham complex, the differential mapping $\Div_F:\kRF{\ell}\to\Poly^{\ell-1}(F)$ is an isomorphism.

For any $T\in\t_h$ now, we introduce the following subspaces of $\bPoly^{\ell}(T)$, $\ell\in\Natu$:
\begin{equation*}
  \begin{aligned}
    \pGT{\ell}&\defi\Grad\big(\Poly^{\ell+1}(T)\big),\qquad&\kGT{\ell}&\defi\bPoly^{\ell-1}(T){\times}(\vec{x}-\vec{x}_T),\\
    \pRT{\ell}&\defi\Curl\big(\bPoly^{\ell+1}(T)\big),\qquad&\kRT{\ell}&\defi\Poly^{\ell-1}(T)(\vec{x}-\vec{x}_T).
  \end{aligned}
\end{equation*}
The two polynomial spaces $\kGT{\ell}$ and $\kRT{\ell}$ are the Koszul complements of $\pGT{\ell}$ and $\pRT{\ell}$, respectively. By the homotopy formula (cf.~\cite[Thm.~7.1]{Arnol:18}), the two following (non $\bLL[T]$-orthogonal) polynomial decompositions hold true:
\begin{equation} \label{eq:decomp.T}
  \bPoly^{\ell}(T)=\pGT{\ell}\oplus\kGT{\ell}=\pRT{\ell}\oplus\kRT{\ell}.
\end{equation}
By exactness of the polynomial de Rham complex, the differential mappings $\Curl:\kGT{\ell}\to\pRT{\ell-1}$ and $\Div:\kRT{\ell}\to\Poly^{\ell-1}(T)$ are isomorphisms. 

\begin{remark}[Inverses of cell-wise differential mappings] \label{re:inv}
  By the inverse mapping theorem, we know that the inverse mappings $\Curl^{-1}\defi(\Curl)^{-1}:\pRT{\ell-1}\to\kGT{\ell}$ and $\vec{\Div}^{-1}\defi(\Div)^{-1}:\Poly^{\ell-1}(T)\to\kRT{\ell}$ are bounded. From a quantitative viewpoint, one can show that
  \begin{equation} \label{eq:nim}
    \|\Curl^{-1}\vec{c}\|_{0,T}\lesssim h_T\|\vec{c}\|_{0,T}\quad\forall\vec{c}\in\pRT{\ell-1},\qquad\|\vec{\Div}^{-1} d\|_{0,T}\lesssim h_T\|d\|_{0,T}\quad\forall d\in\Poly^{\ell-1}(T).
  \end{equation}
  This has been proven for general polyhedral cells by a transport/scaling argument in~\cite[Lem.~9]{DPDro:23}. For star-shaped polyhedral cells, it is possible to obtain explicit multiplicative constants (cf.~\cite[Lem.~2.2]{CDPLe:22} for $\Curl^{-1}$, and Appendix~\ref{ap:nim} for $\vec{\Div}^{-1}$).
\end{remark}

For all $T\in\t_h$ and all $F\in\f_T$, since for $v:T\to\Comp$ we have $\normal_F{\times}((\Grad v)_{\mid F}{\times}\normal_F)=\Grad_F(v_{\mid F})$ (where $\normal_F{\times}((\Grad v)_{\mid F}{\times}\normal_F)$ is identified to its two-dimensional tangential counterpart), there holds
\begin{equation} \label{eq:rttr}
  \pGT{\ell}_{\mid F}{\times}\normal_F=\pRF{\ell},
\end{equation}
where vectors in $\pGT{\ell}_{\mid F}{\times}\normal_F$ are identified to their two-dimensional counterparts.
In the same vein, since for $\vec{v}:T\to\Comp^3$ we have $(\Curl\vec{v})_{\mid F}{\cdot}\normal_F=\Div_F(\vec{v}_{\mid F}{\times}\normal_F)$ (where $\vec{v}_{\mid F}{\times}\normal_F$ is identified to its two-dimensional tangential counterpart), and $\Div_F\big(\bPoly^{\ell+1}(F)\big)=\Poly^{\ell}(F)$ (by~\eqref{eq:decomp.F} with $\ell\leftarrow\ell+1$ and recalling that $\Div_F:\kRF{\ell+1}\to\Poly^{\ell}(F)$ is an isomorphism), there holds
\begin{equation} \label{eq:ntr}
  \pRT{\ell}_{\mid F}{\cdot}\normal_F=\Poly^{\ell}(F).
\end{equation}

For the $3$-variate $\Comp$-valued polynomial space $\Poly^\ell$, and $\Comp^3$-valued polynomial space $\bPoly^{\ell}$,
we introduce the following broken versions:
\begin{align*}
  \Poly^\ell(\t_h)\defi&\{v\in\LL\mid v_{\mid T}\in\Poly^\ell(T)\,\forall T\in\t_h\},\\
  \bPoly^\ell(\t_h)\defi&\{\vec{v}\in\bLL\mid\vec{v}_{\mid T}\in\bPoly^\ell(T)\,\forall T\in\t_h\}.
\end{align*}
We classically define on $\Poly^\ell(\t_h)$ the broken gradient operator $\Grad_h$, and on $\bPoly^{\ell}(\t_h)$ the broken rotational and divergence operators $\Curl_h$ and $\Div_h$.
We let $\pi^\ell_h$ (resp.~$\vec{\pi}^\ell_h$) be the $\LL$-orthogonal (resp.~$\bLL$-orthogonal) projector onto $\Poly^\ell(\t_h)$ (resp.~$\bPoly^{\ell}(\t_h)$).

\subsection{$\vec{H}(\Curl)\cap\vec{H}(\Div_\eta)$-like hybrid spaces}

Let $\ell\in\Natu$ be a given polynomial degree. We define the following discrete counterpart of the space $\cweber=\Hcurl\cap\Hdiv$:
\begin{equation} \label{def:DOFs}
    \dweber \defi\left\{ \vh \defi \Big((\vT)_{T\in\t_h},(\vFt)_{F\in\f_h},(\vFn)_{F\in\f_h}\Big)
     \st 
     \begin{alignedat}{2}
	\vT&\in \bPoly^{\ell}(T) &\quad& \forall T\in\t_h
  		\\
  		\vFt&\in \qRF{\ell} &\quad& \forall F\in\f_h
  		\\
  		\vFn& \in \Poly^{\ell}(F) &\quad& \forall F\in\f_h
      \end{alignedat}
     \right\},
\end{equation}
where the (possibly trimmed) polynomial space $\qRF{\ell}$ satisfies
\begin{equation} \label{eq:qRF}
  \pRF{\ell}\subseteq\qRF{\ell}\subseteq\bPoly^{\ell}(F).
\end{equation} 
We let $\vec{\pi}_{\vec{\mathcal{Q}},F}^{\ell}$ denote the $\bLL[F]$-orthogonal projector onto $\qRF{\ell}$.
In~\eqref{def:DOFs}, $\vFt$ stands for the rotated tangential trace unknown, whereas $\vFn$ stands for the $\eta$-weighted normal trace unknown.

Given a mesh cell $T\in\t_h$, we denote by $\dweberT$ the restriction of $\dweber$ to $T$, and by
$$\vTF\defi\big(\vT,(\vFt)_{F\in\f_T},(\vFn)_{F\in\f_T}\big)\in\dweberT$$
the restriction of the generic element $\vh\in\dweber$ to the cell $T$. 
For $\vh\in\dweber$, we let $\vTh$ (not underlined) be the broken polynomial vector field in $\bPoly^\ell(\t_h)$ such that
$$(\vTh)_{\mid T} \defi \vT \qquad\forall T\in\t_h.$$
Also, we let $\vec{\gamma}_{\vec{\tau},h}:\dweber\to\bLL[\Gamma]$ denote the discrete rotated tangential trace operator, and $\gamma_{n,h}:\dweber\to\LL[\Gamma]$ be the discrete $\eta$-weighted normal trace operator such that, for all $\vh\in\dweber$,
\begin{equation} \label{eq:trop}
  \vec{\gamma}_{\vec{\tau},h}(\vh)_{\mid F}\defi\vFt\qquad\text{and}\qquad\gamma_{n,h}(\vh)_{\mid F}\defi\vFn\qquad\forall F\in\f_h^\partial.
\end{equation}
Note that, for almost every $\vec{x}\in\Gamma$, $\vec{\gamma}_{\vec{\tau},h}(\vh)(\vec{x})$ is a vector in $\Comp^2$.
The discrete trace operators enable us to define the following subspaces of $\dweber$, which are the discrete counterparts of $\cfweber$ and $\csweber$, respectively defined in~\eqref{eq:cfweb} and~\eqref{eq:csweb}:
\begin{equation} \label{eq:dweber}
  \dfweber\defi\left\{\vh\in\dweber\mid\vec{\gamma}_{\vec{\tau},h}(\vh)\equiv\vec{0}\right\},\qquad\dsweber\defi\left\{\vh\in\dweber\mid\gamma_{n,h}(\vh)\equiv 0\right\}.
\end{equation}

For any $T\in\t_h$, we introduce the following local Hermitian, positive semi-definite, sesquilinear forms: for all $\wTF,\vTF\in\dweberT$,
\begin{equation} \label{eq:stab}
  \begin{alignedat}{1}
    s_{\Curl,T}(\wTF,\vTF)&\defi\sum_{F\in\f_T}h_F^{-1}\left(\big(\vec{\pi}^{\ell}_{\vec{\mathcal{Q}},F}\big(\vec{w}_{T\mid F}{\times}\normal_F\big) - \wFt\big),\big(\vec{\pi}^{\ell}_{\vec{\mathcal{Q}},F}\big(\vec{v}_{T\mid F}{\times}\normal_F\big) - \vFt\big)\right)_F,\\
    s_{\Div,T}(\wTF,\vTF)&\defi\sum_{F\in\f_T}h_F^{-1}\Big(\big(\eta_T\vec{w}_{T\mid F}{\cdot}\normal_F - \wFn\big),\big(\eta_T\vec{v}_{T\mid F}{\cdot}\normal_F - \vFn\big)\Big)_F,
  \end{alignedat}
\end{equation}
where $\vec{w}_{T\mid F}{\times}\normal_F$ and $\vec{v}_{T\mid F}{\times}\normal_F$ are identified to their two-dimensional (tangential) counterparts. Based on~\eqref{eq:stab}, we next define $\vec{H}(\Curl)$- and $\vec{H}(\Div_\eta)$-like hybrid semi-norms: for all $\vh\in\dweber$,
\begin{equation} \label{eq:senocu}
  |\vh|_{\Curl,h}^2\defi\sum_{T\in\t_h}\left(\|\Curl\vT\|_{0,T}^2+s_{\Curl,T}(\vTF,\vTF)\right),
\end{equation}
and
\begin{equation} \label{eq:senodi}
  |\vh|_{\Div,h}^2\defi\sum_{T\in\t_h}\left(\|\Div(\eta_T\vT)\|_{0,T}^2+s_{\Div,T}(\vTF,\vTF)\right).
\end{equation}
Finally, we equip the discrete space $\dweber$ with the following weighted norm:
\begin{equation} \label{eq:no}
  \|\vh\|_{\vec{X},h}^2\defi\|\eta^{\frac12}\vTh\|_0^2+\eta_\flat^{-1}|\vh|_{\Div,h}^2+\eta_\sharp|\vh|_{\Curl,h}^2.
\end{equation}

Notice in~\eqref{eq:senocu} the deliberate presence, within the least-squares face penalty terms of $s_{\Curl,T}$, of the tangential face projection $\vec{\pi}^{\ell}_{\vec{\mathcal{Q}},F}\big(\vec{v}_{T\mid F}{\times}\normal_F\big)$. As will be made precise in Lemma~\ref{le:opc}, this projection aims at guaranteeing the optimal (polynomial) consistency of the local Hermitian form $s_{\Curl,T}$, making of the latter a ready-to-use (optimal) stabilization for hybrid schemes. It is quite crucial to note that the introduction of this face projection does not jeopardize stability, as is made clear in Lemma~\ref{le:fuco} below. This latter fact essentially relies on the assumption~\eqref{eq:qRF} on the tangential face space. Similar considerations actually apply to~\eqref{eq:senodi} and $s_{\Div,T}$ but with the difference that, since the normal face space is chosen to be $\Poly^{\ell}(F)$ and $\eta$ is piecewise constant on $\t_h$, then $\pi^{\ell}_F(\eta_T\vec{v}_{T\mid F}{\cdot}\normal_F)=\eta_T\vec{v}_{T\mid F}{\cdot}\normal_F$ for $\vT\in\bPoly^{\ell}(T)$. The heuristics behind the choice of face spaces is further motivated in Remark~\ref{rem:fac.spa}.
\begin{lemma}[Full control of tangential jumps] \label{le:fuco}
  For all $\vh\in\dweber$, and all $T\in\t_h$, the following holds true:
  \begin{equation} \label{eq:withoutproj}
    \left(\sum_{F\in\f_T}h_F^{-1}\|\vec{v}_{T\mid F}{\times}\normal_F-\vFt\|_{0,F}^2\right)^{\nicefrac12}\lesssim\|\Curl\vT\|_{0,T}+s_{\Curl,T}(\vTF,\vTF)^{\nicefrac12}.
  \end{equation}
\end{lemma}
\begin{proof}
  By Remark~\ref{re:inv}, notice that $\big(\vT-\Curl^{-1}(\Curl\vT)\big)\in\pGT{\ell}$. Hence, by~\eqref{eq:rttr}, for any $F\in\f_T$,
  $$\big(\vec{v}_{T\mid F}{\times}\normal_F-\Curl^{-1}(\Curl\vT)_{\mid F}{\times}\normal_F\big)\in\pRF{\ell}.$$
  Since $\pRF{\ell}\subseteq\qRF{\ell}$ owing to~\eqref{eq:qRF}, we then infer that
  $$\vec{v}_{T\mid F}{\times}\normal_F-\vFt=\big(\vec{I}-\vec{\pi}^{\ell}_{\vec{\mathcal{Q}},F}\big)\big(\Curl^{-1}(\Curl\vT)_{\mid F}{\times}\normal_F\big)+\vec{\pi}^\ell_{\vec{\mathcal{Q}},F}\big(\vec{v}_{T\mid F}{\times}\normal_F\big)-\vFt.$$
  By the triangle inequality, and the $\bLL[F]$-optimality of the orthogonal projector $\vec{\pi}^{\ell}_{\vec{\mathcal{Q}},F}$, we get
  \begin{multline*}
    h_F^{-\frac12}\|\vec{v}_{T\mid F}{\times}\normal_F-\vFt\|_{0,F}\leq h_F^{-\frac12}\|\Curl^{-1}(\Curl\vT)_{\mid F}{\times}\normal_F\|_{0,F}\\+h_F^{-\frac12}\|\vec{\pi}^{\ell}_{\vec{\mathcal{Q}},F}\big(\vec{v}_{T\mid F}{\times}\normal_F\big)-\vFt\|_{0,F},
  \end{multline*}
  which, in turn, by a discrete trace inequality (cf.~e.g.~\cite[Lem.~1.32]{DPDro:20}) combined to mesh regularity, the estimate~\eqref{eq:nim} on the local norm of $\Curl^{-1}$, and the expression~\eqref{eq:stab} of $s_{\Curl,T}$, yields~\eqref{eq:withoutproj}.
\end{proof}
\begin{remark}[Face spaces] \label{rem:fac.spa}
  For a cell unknown $\vT$ in $\bPoly^{\ell}(T)$, the choice of the tangential face space $\qRF{\ell}$ satisfying~\eqref{eq:qRF} is driven by the following stability argument. Locally to a mesh cell $T\in\t_h$, the face projection $\vec{\pi}^{\ell}_{\vec{\mathcal{Q}},F}\big(\vec{v}_{T\mid F}{\times}\normal_F\big)$ for $F\in\f_T$ must enable to control, via least-squares face penalty, (at least) the rotated tangential trace of the $\pGT{\ell}$-part of $\vT$ (i.e.~$\vT-\Curl^{-1}(\Curl\vT)$). As a matter of fact, $\pGT{\ell}$ is precisely the kernel of the rotational operator in $\bPoly^{\ell}(T)$. Since, by~\eqref{eq:rttr}, $\pGT{\ell}_{\mid F}{\times}\normal_F=\pRF{\ell}$, this motivates the choice of $\qRF{\ell}$ such that $\pRF{\ell}\subseteq\qRF{\ell}\subseteq\bPoly^\ell(F)$, the space $\pRF{\ell}$ being the smallest admissible set ensuring stability.
  If one wants to adapt the argument to the normal face space, one must now control with a face projection of $\eta_T\vec{v}_{T\mid F}{\cdot}\normal_F$ (at least) the normal trace of the $\pRT{\ell}$-part of $\eta_T\vT$ (i.e.~$\eta_T\vT-\vec{\Div}^{-1}(\Div(\eta_T\vT))$). As a matter of fact, $\eta_T$ is a constant, so that $\eta_T\vT\in\bPoly^{\ell}(T)$, and $\pRT{\ell}$ is precisely the kernel of the divergence operator in $\bPoly^{\ell}(T)$. Since, by~\eqref{eq:ntr}, $\pRT{\ell}_{\mid F}{\cdot}\normal_F=\Poly^{\ell}(F)$, this means that choosing $\Poly^{\ell}(F)$ as the normal face space is actually optimal.
\end{remark}

For $q>2$, we define the following global reduction operator: $\underline{\mathbbmsl{I}}^{\ell}_h:\cweber\cap\vec{L}^q(\Omega)\to\dweber$ is such that, for all $\vec{v}\in\cweber\cap\vec{L}^q(\Omega)$,
\begin{equation} \label{def:Idch}
    \underline{\mathbbmsl{I}}^{\ell}_h(\vec{v}) \defi \left(
    \big(\vec{\pi}_{T}^{\ell}(\vec{v}_{\mid T})\big)_{T\in\t_h},
    \big(\vec{\pi}_{\vec{\mathcal{Q}},F}^{\ell}(\vec{v}_{\mid F}{\times}\normal_F)\big)_{F\in\f_h},\big(\pi_F^{\ell}((\eta\vec{v})_{\mid F}{\cdot}\normal_F)\big)_{F\in\f_h}
    \right),
\end{equation}
where $\vec{v}_{\mid F}{\times}\normal_F$ is identified to its two-dimensional (tangential) counterpart.
\begin{remark}[Regularity]
  Following~\cite{ErnGu:22}, for $\eta$ satisfying~\eqref{eq:eta.c}, the regularity $\vec{v}\in\cweber\cap\vec{L}^q(\Omega)$ is sufficient to give a (weak) meaning to the face polynomial projections in~\eqref{def:Idch}. Besides, this regularity is typically satisfied by the solutions to the Maxwell system with piecewise smooth parameter $\eta$ as in~\eqref{eq:eta.d}. Indeed, functions in $\cfweber$ or $\csweber$ with piecewise smooth $\eta$ belong to $\vec{H}^s(\Omega)$ for some $s>0$ (cf.~\cite{CoDaN:99,Jochm:99,BGLud:13}), i.e.~by Sobolev embedding, to $\vec{L}^q(\Omega)$ for some $q>2$.
\end{remark}
\noindent
For all $T\in\t_h$, we also introduce the local reduction operator: $\underline{\mathbbmsl{I}}^{\ell}_T:\Hcurl[T]\cap\Hdiv[T]\cap\vec{L}^q(T)\to\dweberT$ is such that, for all $\vec{v}\in\Hcurl[T]\cap\Hdiv[T]\cap\vec{L}^q(T)$,
\begin{equation} \label{def:Idch.loc}
    \underline{\mathbbmsl{I}}^{\ell}_T(\vec{v}) \defi \left(
    \vec{\pi}_{T}^{\ell}(\vec{v}),
    \big(\vec{\pi}_{\vec{\mathcal{Q}},F}^{\ell}(\vec{v}_{\mid F}{\times}\normal_F)\big)_{F\in\f_T},\big(\pi_F^{\ell}(\eta_T\vec{v}_{\mid F}{\cdot}\normal_F)\big)_{F\in\f_T}
    \right).
\end{equation}
The proof of the next lemma is straightforward based on the definitions~\eqref{eq:stab} and~\eqref{def:Idch.loc}.
\begin{lemma}[Optimal polynomial consistency] \label{le:opc}
  Let $T\in\t_h$. For all $\vec{p}\in\bPoly^{\ell}(T)$, letting $\underline{\mathbbmsl{p}}_T\in\dweberT$ be such that $\underline{\mathbbmsl{p}}_T\defi\underline{\mathbbmsl{I}}^{\ell}_T(\vec{p})$, the following holds true:
  \begin{itemize}
    \item[$\bullet$] $s_{\Curl,T}(\underline{\mathbbmsl{p}}_T,\vTF)=0$ \,for all $\vTF\in\dweberT$;
    \item[$\bullet$] $s_{\Div,T}(\underline{\mathbbmsl{p}}_T,\vTF)=0$ \,for all $\vTF\in\dweberT$.
  \end{itemize}
\end{lemma}
\noindent
The important Lemma~\ref{le:opc} above guarantees that the stabilization Hermitian forms~\eqref{eq:stab} shall provide superconvergence when used within hybrid numerical schemes (cf.~\cite{CQSSo:17,CDPLe:22}). In the literature, this kind of stabilization is often referred to as of Lehrenfeld--Sch\"oberl type~\cite{LeSch:16} (see also~\cite{CEPig:21}).

\section{Hybrid Weber inequalities} \label{se:weber}

Let $\ell\in\Natu$ be a given polynomial degree.

\subsection{First Weber inequality}

Recall the definitions~\eqref{eq:dweber} of the discrete space $\dfweber$,~\eqref{eq:senocu}--\eqref{eq:senodi} of the hybrid semi-norms $|{\cdot}|_{\Curl,h}$ and $|{\cdot}|_{\Div,h}$, and~\eqref{eq:trop} of the discrete $\eta$-weighted normal trace operator $\gamma_{n,h}$.
\begin{theorem}[First hybrid Weber inequality] \label{thm:fweber}
  There exists $c_{W,1}>0$ independent of $h$ and $\eta$ such that, for all $\vh\in\dfweber$, one has
  \begin{equation} \label{thm:max.ineq.eq1}
    \|\eta^{\frac12}\vTh\|_0 \leq c_{W,1}\left(\eta_\flat^{-\frac12}|\vh|_{\Div,h}+\eta_\sharp^{\frac12}|\vh|_{\Curl,h}+\eta_\flat^{-\frac12}\kappa_\eta^{\frac12}\Bigg(\sum_{j=1}^{\beta_2}\big|\big(\gamma_{n,h}(\vh),1\big)_{\Gamma_j}\big|^2\Bigg)^{\nicefrac12}\right).
  \end{equation}
\end{theorem}
\begin{proof}
  Since $\vTh\in\bPoly^{\ell}(\t_h)\subset\bLL$, by the first Helmholtz--Hodge decomposition~\eqref{eq:helm1}, there exist $\varphi\in H^1_0(\Omega)$, $\vec{\psi}\in\vec{H}^1(\Omega)\cap\vec{L}^2_{\vec{0}}(\Omega)$, and $\vec{w}\in\fharmonic$ (writing $\vec{w}=\Grad\omega$, with $\omega\in H^1(\Omega)$ such that $\omega\defi\sum_{j=1}^{\beta_2}\alpha_j\omega_j$), such that
  \begin{equation} \label{eq:hodge}
    \vTh = \Grad \varphi + \frac{\eta_\sharp}{\eta}\Curl \vec{\psi} + \vec{w},
  \end{equation}
  with $\|\vec{\psi}\|_1\leq C_{\Omega,1}\|\Curl\vec{\psi}\|_0$ by~\eqref{hh.reg1}.
  From the decomposition~\eqref{eq:hodge}, we directly infer that
  \begin{equation}\label{eq:max.proof:norm.decomp1} 
    \|\eta^{\frac12}\vTh\|_0^2 =
    (\eta\vTh,\Grad \varphi)_\Omega
    +\eta_\sharp\big(\vTh,\Curl\vec{\psi}\big)_\Omega
    +\big(\eta\vTh, \Grad \omega \big)_\Omega
    \ifed \mathcal{I}_1 + \eta_\sharp\mathcal{I}_2 +\mathcal I_3.
  \end{equation}
  Let us now estimate the three addends in the right-hand side of~\eqref{eq:max.proof:norm.decomp1}.

  For $\mathcal{I}_1$, first remark that
  $$\sum_{T\in\t_h}\sum_{F\in\f_T}\varepsilon_{T,F}\big(\vFn,\varphi_{\mid F}\big)_F=0,$$
  as a consequence of the fact that $\varphi\in H^1_0(\Omega)$ (so that $\varphi$ is single-valued at interfaces and vanishes on boundary faces) and that, for any $F\subset\partial T^+\cap\partial T^-\in\f_h^\circ$, $\varepsilon_{T^+,F}+\varepsilon_{T^-,F}=0$.
  Hence, performing cell-by-cell integration by parts, we get
  \begin{align*}
    \mathcal{I}_{1}
    &=\sum_{T \in \t_h}\big(\eta_T\vT, \Grad \varphi)_T
    =\sum_{T\in\t_h} \bigg(
    -\big(\Div (\eta_T\vT), \varphi \big)_T
    +\sum_{F\in\f_T} \big( \eta_T\vec{v}_{T\mid F}{\cdot}\normal_{T,F}, \varphi_{\mid F} \big)_F
    \bigg)\\
    &=\sum_{T\in\t_h} -\big(\Div (\eta_T\vT), \varphi \big)_T
    +\sum_{T\in\t_h} \sum_{F\in\f_T} \varepsilon_{T,F}\big( \eta_T\vec{v}_{T\mid F}{\cdot}\normal_F - \vFn, \varphi_{\mid F} \big)_F.
  \end{align*}
  Repeated applications of the Cauchy--Schwarz inequality then yield
  $$\mathcal{I}_{1}\leq|\vh|_{\Div,h}\left(\|\varphi\|^2_0+\sum_{T\in\t_h} \sum_{F\in\f_T}h_F\|\varphi_{\mid F}\|^2_{0,F}\right)^{\nicefrac12}.$$
  Applying a continuous trace inequality (cf.~e.g.~\cite[Lem.~1.31]{DPDro:20}) combined with mesh regularity, using the fact that $h_T\leq{\rm diam}(\Omega)$ for all $T\in\t_h$, and concluding with the Poincar\'e inequality applied to $\varphi\in H^1_0(\Omega)$, we finally infer that
  \begin{equation} \label{eq:I1}
    \mathcal{I}_{1}\lesssim|\vh|_{\Div,h}\|\Grad\varphi\|_0.
  \end{equation}

  Let us now turn to the estimation of $\mathcal{I}_2$. In what follows, following Remark~\ref{re:tvf}, all tangential vector fields are pointwise identified to their counterparts in $\Comp^2$. First, notice that
  $$\sum_{T\in\t_h}\sum_{F\in\f_T}\varepsilon_{T,F}\big(\vFt,\normal_{F}{\times}(\vec{\psi}_{\mid F}{\times}\normal_{F}) \big)_F=0,$$
  as a consequence of the fact that $\vec{\psi}\in \vec{H}^1(\Omega)$ (so that all components of $\vec{\psi}$ are square-integrable and single-valued at interfaces), that $\vh\in\dfweber$ (so that $\vFt\equiv\vec{0}$ for all $F\in\f_h^\partial$), and that, for any $F\subset\partial T^+\cap\partial T^-\in\f_h^\circ$, $\varepsilon_{T^+,F}+\varepsilon_{T^-,F}=0$.
  Hence, performing cell-by-cell integration by parts yields
\begin{align*}
\mathcal{I}_2 &=
\sum_{T\in\t_h}\big(\vT,\Curl\vec{\psi}\big)_T
=\sum_{T\in\t_h} \bigg(\big(\Curl \vT,\vec{\psi}\big)_T
+\sum_{F\in\f_T}\big(\vec{v}_{T\mid F}{\times}\normal_{T,F},\normal_{F}{\times}(\vec{\psi}_{\mid F}{\times}\normal_{F}) \big)_F \bigg)
\\
&=\sum_{T\in\t_h} \big(\Curl \vT,\vec{\psi}\big)_T
+ \sum_{T\in\t_h} \sum_{F\in\f_T}\varepsilon_{T,F}\big(\vec{v}_{T\mid F}{\times}\normal_F-\vFt,\normal_{F}{\times}(\vec{\psi}_{\mid F}{\times}\normal_{F}) \big)_F.
\end{align*}
By repeated applications of the Cauchy--Schwarz inequality, we then get
\begin{multline} \label{eq:I2b}
\mathcal{I}_2
\leq \Bigg(
\|\Curl_h\!\vTh\|_0^2+\sum_{T\in\t_h}\sum_{F\in\f_T} h_F^{-1}\|\vec{v}_{T\mid F}{\times}\normal_F-\vFt\|_{0,F}^2
\Bigg)^{\nicefrac{1}{2}}
\\\times\Bigg(
\|\vec{\psi}\|_0^2+\sum_{T\in\t_h}\sum_{F\in\f_T} h_F\|\normal_{F}{\times}(\vec{\psi}_{\mid F}{\times}\normal_{F})\|_{0,F}^2
\Bigg)^{\nicefrac{1}{2}}.
\end{multline}
Let us focus on the first factor in the right-hand side of~\eqref{eq:I2b}. By the estimate~\eqref{eq:withoutproj}, and the definition~\eqref{eq:senocu} of $|{\cdot}|_{\Curl,h}$, we infer that
$$\Bigg(
\|\Curl_h\!\vTh\|_0^2+\sum_{T\in\t_h}\sum_{F\in\f_T} h_F^{-1}\|\vec{v}_{T\mid F}{\times}\normal_F-\vFt\|_{0,F}^2
\Bigg)^{\nicefrac{1}{2}}\lesssim|\vh|_{\Curl,h}.$$
For the second factor in the right-hand side of~\eqref{eq:I2b}, applying a continuous trace inequality combined with mesh regularity, using the fact that $h_T\leq {\rm diam}(\Omega)$ for all $T\in\t_h$, and concluding with the estimate $\|\vec{\psi}\|_1\leq C_{\Omega,1}\|\Curl\vec{\psi}\|_0$, there holds
\begin{equation*}
\left(
\|\vec{\psi}\|_0^2
+\sum_{T\in\t_h}\sum_{F\in\f_T} h_F\|\normal_{F}{\times}(\vec{\psi}_{\mid F}{\times}\normal_{F})\|_{0,F}^2
		\right)^{\nicefrac{1}{2}}
		\lesssim\|\Curl\vec{\psi}\|_0.
\end{equation*}
Collecting the different results, we finally infer that
\begin{equation} \label{eq:I2}
  \mathcal{I}_2\lesssim |\vh|_{\Curl,h}\|\Curl\vec{\psi}\|_0.
\end{equation}

We are now left with estimating $\mathcal{I}_3$. First, remark that
$$\sum_{T\in\t_h}\sum_{F\in\f_T}\varepsilon_{T,F}\big(\vFn,\omega_{\mid F}\big)_F=\sum_{j=1}^{\beta_2}\big(\gamma_{n,h}(\vh) , 1 \big)_{\Gamma_j}\overline{\alpha_j},$$
as a consequence of the fact that $\omega\in H^1(\Omega)$ (so that $\omega$ is single-valued at interfaces), that $\omega=\sum_{j=1}^{\beta_2}\alpha_j\omega_j$ with $\omega_{j\mid\Gamma_{j'}}=\delta_{jj'}$ for all $j'\in\{0,\ldots,\beta_2\}$ (see~\eqref{eq:harmonic1}), and that, for any $F\subset\partial T^+\cap\partial T^-\in\f_h^{\circ}$, $\varepsilon_{T^+,F}+\varepsilon_{T^-,F}=0$, whereas for any $F\subset\partial T\cap\Gamma\in\f_h^{\partial}$, $\varepsilon_{T,F}=1$.
Hence, performing cell-by-cell integration by parts, we get
\begin{align*}
\mathcal{I}_{3}
&=\sum_{T \in \t_h}\big(\eta_T\vT, \Grad \omega\big)_T
=\sum_{T\in\t_h} \left(
-\big(\Div (\eta_T\vT), \omega \big)_T
+\sum_{F\in\f_T} \big( \eta_T\vec{v}_{T\mid F}{\cdot}\normal_{T,F}, \omega_{\mid F} \big)_F
\right)\\
&=\sum_{T\in\t_h} -\big(\Div (\eta_T\vT), \omega \big)_T
+\sum_{T\in\t_h} \sum_{F\in\f_T } \varepsilon_{T,F}\big( \eta_T\vec{v}_{T\mid F}{\cdot}\normal_{F} - \vFn, \omega_{\mid F} \big)_F \\
&+\sum_{j=1}^{\beta_2}\big(\gamma_{n,h}(\vh) , 1 \big)_{\Gamma_j}\overline{\alpha_j}.
\end{align*}
Repeated applications of the Cauchy--Schwarz inequality then yield
$$\mathcal{I}_{3}\leq|\vh|_{\Div,h}\left(\|\omega\|^2_0+\sum_{T\in\t_h} \sum_{F\in\f_T}h_F\|\omega_{\mid F}\|^2_{0,F}\right)^{\nicefrac12}+\Bigg(\sum_{j=1}^{\beta_2}\big|\big(\gamma_{n,h}(\vh) , 1 \big)_{\Gamma_j}\big|^2\Bigg)^{\nicefrac12}|\vec{\alpha}_{\vec{w}}|,$$
where $|\vec{\alpha}_{\vec{w}}|$ denotes the Euclidean norm of $\vec{\alpha}_{\vec{w}}\defi(\alpha_j\in\Comp)_{j\in\{1,\dots,\beta_2\}}$.
Proceeding as in the continuous setting to estimate $|\vec{\alpha}_{\vec{w}}|$ (cf.~\eqref{eq:alpha1}), and as for the term $\mathcal{I}_1$ to estimate the $\omega$-factor (recall that $\omega_{\mid\Gamma_0}=0$), we finally infer that
\begin{equation} \label{eq:I3}
  \mathcal{I}_3\lesssim |\vh|_{\Div,h}\|\Grad\omega\|_0+\eta_\flat^{-\frac12}\kappa_\eta^{\frac12}\Bigg(\sum_{j=1}^{\beta_2}\big|\big(\gamma_{n,h}(\vh) , 1 \big)_{\Gamma_j}\big|^2\Bigg)^{\nicefrac12}\|\eta^{\frac12}\Grad\omega\|_0.
\end{equation}

By the $\vec{L}^2_\eta(\Omega)$-orthogonality of the decomposition~\eqref{eq:hodge}, there holds $\|\Grad\varphi\|_0\leq\eta_\flat^{-\frac12}\|\eta^{\frac12}\vTh\|_0$, $\|\Curl\vec{\psi}\|_0\leq\eta_\sharp^{-\frac12}\|\eta^{\frac12}\vTh\|_0$, and $\|\Grad\omega\|_0\leq\eta_\flat^{-\frac12}\|\eta^{\frac12}\Grad\omega\|_0\leq\eta_\flat^{-\frac12}\|\eta^{\frac12}\vTh\|_0$.
Plugging the three estimates~\eqref{eq:I1},~\eqref{eq:I2}, and~\eqref{eq:I3} into~\eqref{eq:max.proof:norm.decomp1}, we eventually obtain~\eqref{thm:max.ineq.eq1}.
\end{proof}

\begin{remark}[Useful variant] \label{rem:var1d}
  Let us introduce the following hybrid counterpart of the space $\Hcurl$:
  \begin{equation} \label{def:DOFs.c}
    \dweberc \defi\left\{ \vhc \defi \Big((\vT)_{T\in\t_h},(\vFt)_{F\in\f_h}\Big)
     \st 
     \begin{alignedat}{2}
	\vT&\in \bPoly^{\ell}(T) &\quad& \forall T\in\t_h
  		\\
  		\vFt&\in \qRF{\ell} &\quad& \forall F\in\f_h
      \end{alignedat}
     \right\},
  \end{equation}
  along with its subspace $\dfweberc\defi\left\{\vhc\in\dweberc\mid\vFt\equiv\vec{0}\,\forall F\in\f_h^{\partial}\right\}$ (hybrid counterpart of $\Hzcurl$).
  Following Remark~\ref{rem:var1}, assume that $\vhc\in\dfweberc$ is such that
  \begin{equation} \label{eq:ortho1}
    (\eta\vTh,\vec{z})_{\Omega}=0\qquad\forall\vec{z}\in\Grad\big(H^1_0(\Omega)\big)\overset{\perp_\eta}{\oplus}\fharmonic.
  \end{equation}
  Then, following the proof of Theorem~\ref{thm:fweber}, and (seamlessly) restricting the definition~\eqref{eq:senocu} of $|{\cdot}|_{\Curl,h}$ to vectors in $\dweberc$, one can show that
  \begin{equation} \label{eq:usweb1}
    \|\eta^{\frac12}\vTh\|_0\lesssim\eta_\sharp^{\frac12}|\vhc|_{\Curl,h},
  \end{equation}
  which is a complex-valued, variable-$\eta$ extension to possibly non-trivial topologies of~\cite[Thm.~2.1]{CDPLe:22}. We refer the reader to~\cite{CDPLe:22} for an example, in the magnetostatics framework on trivial domains, of how the orthogonality condition~\eqref{eq:ortho1} may be (possibly approximately) imposed at the discrete level on general polyhedral mesh families.
\end{remark}

\subsection{Second Weber inequality}

Recall the definitions~\eqref{eq:dweber} of the discrete space $\dsweber$,~\eqref{eq:senocu}--\eqref{eq:senodi} of the hybrid semi-norms $|{\cdot}|_{\Curl,h}$ and $|{\cdot}|_{\Div,h}$, and~\eqref{eq:trop} of the discrete $\eta$-weighted normal trace operator $\gamma_{n,h}$.
\begin{theorem}[Second hybrid Weber inequality] \label{thm:sweber}
  There exists $c_{W,2}>0$ independent of $h$ and $\eta$ such that, for all $\vh\in \dsweber$, one has
  \begin{equation}\label{thm:max.ineq.eq2}
    \|\eta^{\frac12}\vTh\|_0 \leq c_{W,2}\left(\eta_\flat^{-\frac12}|\vh|_{\Div,h}+\eta_\sharp^{\frac12}|\vh|_{\Curl,h}+\eta_\flat^{-\frac12}\kappa_\eta^{\frac12}\Bigg(\sum_{i=1}^{\beta_1}\big|\big(\gamma_{n,h}(\vh) , 1 \big)_{\Sigma_i}\big|^2\Bigg)^{\nicefrac12}\right).
  \end{equation}
\end{theorem}
\begin{proof}
  Since $\vTh\in\bPoly^{\ell}(\t_h)\subset\bLL$, by the second Helmholtz--Hodge decomposition~\eqref{eq:helm2}, there exist $\varphi\in H^1(\Omega)\cap L^2_0(\Omega)$, $\vec{\psi}\in\vec{H}^1(\Omega)\cap\Hzcurl$, and $\vec{w}\in\sharmonic$ (writing $\vec{w}=\check{\Grad\pi}$, with $\pi\in H^1(\hat{\Omega})\cap L^2_0(\hat{\Omega})$ such that $\pi\defi\sum_{i=1}^{\beta_1}\alpha_i\pi_i$), such that
  \begin{equation} \label{eq:hodgeb}
    \vTh = \Grad \varphi + \frac{\eta_\sharp}{\eta}\Curl \vec{\psi} + \vec{w},
  \end{equation}
  with $\|\vec{\psi}\|_1\leq C_{\Omega,2}\|\Curl\vec{\psi}\|_0$ by~\eqref{hh.reg2}.
  From the decomposition~\eqref{eq:hodgeb}, we directly infer that
  \begin{equation}\label{eq:max.proof:norm.decomp2} 
    \|\eta^{\frac12}\vTh\|_0^2 =
    (\eta\vTh,\Grad \varphi)_\Omega
    +\eta_\sharp\big(\vTh,\Curl\vec{\psi}\big)_\Omega
    +\big(\eta\vTh, \Grad \pi \big)_{\hat{\Omega}}
    \ifed \mathcal{I}_1 + \eta_\sharp\mathcal{I}_2 +\mathcal I_3.
  \end{equation}
  Let us now estimate the three addends in the right-hand side of~\eqref{eq:max.proof:norm.decomp2}.

  For $\mathcal{I}_1$ and $\mathcal{I}_2$, it is straightforward to prove, adapting the arguments advocated in the proof of Theorem~\ref{thm:fweber}, that
  \begin{equation} \label{eq:I12}
    \mathcal{I}_1+\eta_\sharp\mathcal{I}_2\lesssim|\vh|_{\Div,h}\|\Grad\varphi\|_0+\eta_\sharp|\vh|_{\Curl,h}\|\Curl\vec{\psi}\|_0.
  \end{equation}

  Let us now focus on $\mathcal{I}_3$. First, remark that
  $$\sum_{T\in\t_h}\sum_{F\in\f_T}\varepsilon_{T,F}\big(\vFn,(\pi_{\mid T})_{\mid F}\big)_F=\sum_{i=1}^{\beta_1}\big(\gamma_{n,h}(\vh) , 1 \big)_{\Sigma_i}\overline{\alpha_i},$$
  as a consequence of the fact that $\pi\in H^1(\hat{\Omega})$ (so that $\pi$ is single-valued at interfaces $F\in\hat{\f}_h^\circ$), that $\pi=\sum_{i=1}^{\beta_1}\alpha_i\pi_i$ with $\llbracket\pi_i\rrbracket_{\Sigma_{i'}}=\delta_{ii'}$ for all $i'\in\{1,\ldots,\beta_1\}$ (see~\eqref{eq:harmonic2}), that $\vh\in\dsweber$ (so that $\vFn=0$ for all $F\in\f_h^\partial$), and that, for any $F\subset\partial T^+\cap\partial T^-\in\f_h^\circ$, $\varepsilon_{T^+,F}+\varepsilon_{T^-,F}=0$ with, for any $i\in\{1,\ldots,\beta_1\}$, and any $F\subset\partial T^+\cap\partial T^-\in\f_{h,\Sigma_i}^\circ$, $\varepsilon_{T^+,F}=1$.
  Hence, performing cell-by-cell integration by parts, we get
\begin{align*}
\mathcal{I}_{3}
&=\sum_{T \in \t_h}\big(\eta_T\vT, \Grad \pi)_T
=\sum_{T\in\t_h} \left(
-\big(\Div (\eta_T\vT), \pi \big)_T
+\sum_{F\in\f_T} \big( \eta_T\vec{v}_{T\mid F}{\cdot}\normal_{T,F}, (\pi_{\mid T})_{\mid F} \big)_F
\right)\\
&=\sum_{T\in\t_h}
-\big(\Div (\eta_T\vT), \pi \big)_T
+\sum_{T\in\t_h} \sum_{F\in\f_T} \varepsilon_{T,F}\big( \eta_T\vec{v}_{T\mid F}{\cdot}\normal_F - \vFn,(\pi_{\mid T})_{\mid F} \big)_F\\
&+\sum_{i=1}^{\beta_1}\big(\gamma_{n,h}(\vh) , 1 \big)_{\Sigma_i}\overline{\alpha_i}.
\end{align*}
Repeated applications of the Cauchy--Schwarz inequality then yield
$$\mathcal{I}_{3}\leq|\vh|_{\Div,h}\left(\|\pi\|^2_{0,\hat{\Omega}}+\sum_{T\in\t_h} \sum_{F\in\f_T}h_F\|(\pi_{\mid T})_{\mid F}\|^2_{0,F}\right)^{\nicefrac12}+\Bigg(\sum_{i=1}^{\beta_1}\big|\big(\gamma_{n,h}(\vh) , 1 \big)_{\Sigma_i}\big|^2\Bigg)^{\nicefrac12}|\vec{\alpha}_{\vec{w}}|,$$
where $|\vec{\alpha}_{\vec{w}}|$ denotes the Euclidean norm of $\vec{\alpha}_{\vec{w}}\defi(\alpha_i\in\Comp)_{i\in\{1,\dots,\beta_1\}}$.
Proceeding as in the continuous setting (cf.~\eqref{eq:alpha2}) to estimate $|\vec{\alpha}_{\vec{w}}|$, and applying to the $\pi$-factor a continuous trace inequality combined with mesh regularity (and the fact that $h_T\leq{\rm diam}(\Omega)$ for all $T\in\t_h$), along with a Poincar\'e--Steklov inequality (recall that $\pi\in H^1(\hat{\Omega})\cap L^2_0(\hat{\Omega})$ and that $\hat{\Omega}$ is assumed to be connected), we finally infer
\begin{equation} \label{eq:I3b}
  \mathcal{I}_3\lesssim |\vh|_{\Div,h}\|\check{\Grad\pi}\|_0+\eta_\flat^{-\frac12}\kappa_\eta^{\frac12}\Bigg(\sum_{i=1}^{\beta_1}\big|\big(\gamma_{n,h}(\vh) , 1 \big)_{\Sigma_i}\big|^2\Bigg)^{\nicefrac12}\|\eta^{\nicefrac12}\check{\Grad\pi}\|_0,
\end{equation}
where we also used that $\|\Grad\pi\|_{0,\hat{\Omega}}=\|\check{\Grad\pi}\|_0$.

By the $\vec{L}^2_\eta(\Omega)$-orthogonality of the decomposition~\eqref{eq:hodgeb}, there holds $\|\Grad\varphi\|_0\leq\eta_\flat^{-\frac12}\|\eta^{\frac12}\vTh\|_0$, $\|\Curl\vec{\psi}\|_0\leq\eta_\sharp^{-\frac12}\|\eta^{\frac12}\vTh\|_0$, and $\|\check{\Grad\pi}\|_0\leq\eta_\flat^{-\frac12}\|\eta^{\frac12}\check{\Grad\pi}\|_0\leq\eta_\flat^{-\frac12}\|\eta^{\frac12}\vTh\|_0$.
Plugging the two estimates~\eqref{eq:I12} and~\eqref{eq:I3b} into~\eqref{eq:max.proof:norm.decomp2}, we eventually obtain~\eqref{thm:max.ineq.eq2}.
\end{proof}

\begin{remark}[Useful variant] \label{rem:var2d}
  Recall the definition~\eqref{def:DOFs.c} of the hybrid counterpart $\dweberc$ of the space $\Hcurl$.
  Following Remark~\ref{rem:var2}, assume that $\vhc\in\dweberc$ is such that
  \begin{equation} \label{eq:ortho2}
    (\eta\vTh,\vec{z})_{\Omega}=0\qquad\forall\vec{z}\in\Grad\big(H^1(\Omega)\cap L^2_0(\Omega)\big)\overset{\perp_\eta}{\oplus}\sharmonic.
  \end{equation}
  Then, revisiting the proof of Theorem~\ref{thm:sweber}, one can show that
  \begin{equation} \label{eq:usweb2}
    \|\eta^{\frac12}\vTh\|_0\lesssim \eta_\sharp^{\frac12}|\vhc|_{\Curl,h}.
  \end{equation}
\end{remark}

\section{Discrete Maxwell compactness} \label{se:maxcom}

Let $\ell,m,p\in\Natu$ be three given polynomial degrees.

\subsection{Discrete differential operators}

Recall the definitions~\eqref{def:DOFs} of the discrete space $\dweber$,~\eqref{eq:stab} of the local Hermitian forms $s_{\Curl,T}$ and $s_{\Div,T}$, and~\eqref{eq:senocu}--\eqref{eq:senodi} of the hybrid semi-norms $|{\cdot}|_{\Curl,h}$ and $|{\cdot}|_{\Div,h}$.

We define the (global) rotational reconstruction operator $\Cmh : \dweber \to \bPoly^m(\t_h)$ through its local restrictions $\CmT: \dweberT \to \bPoly^m(T)$ to any mesh cell $T \in \t_h$: for all $\vTF\in\dweberT$, $\CmT(\vTF)\in\bPoly^m(T)$ is the unique solution to
\begin{equation}\label{def:CTk}
  \big(\CmT(\vTF), \vec{q}\big)_T = (\vT,\Curl \vec{q})_T
-\!\!\!\sum_{F\in\f_T}\!\varepsilon_{T,F}\big(\vFt,\normal_F{\times}(\vec{q}_{\mid F}{\times}\normal_F)\big)_F \qquad\forall\vec{q}\in\bPoly^m(T),
\end{equation}
where, as now standard, $\normal_F{\times}(\vec{q}_{\mid F}{\times}\normal_F)$ is identified to its two-dimensional (tangential) counterpart.
Performing an integration by parts of the term $(\vT,\Curl \vec{q})_T$, and using a discrete trace inequality along with Lemma~\ref{le:fuco}, it is an easy matter to prove that
\begin{equation} \label{eq:bound.c}
  \sum_{T\in\t_h}\left(\|\CmT(\vTF)\|_{0,T}^2+s_{\Curl,T}(\vTF,\vTF)\right)\lesssim|\vh|^2_{\Curl,h}\qquad\forall\vh\in\dweber.
\end{equation}
If $m\geq\ell-1$, and $\qRF{\ell}$ (already satisfying~\eqref{eq:qRF}) also satisfies $\bPoly^{\ell-1}(F)\subset\qRF{\ell}$ for all $F\in\f_h$, one can also prove that $|\vh|^2_{\Curl,h}\lesssim\sum_{T\in\t_h}\left(\|\CmT(\vTF)\|_{0,T}^2+s_{\Curl,T}(\vTF,\vTF)\right)$.

In the same vein, we define the (global) divergence reconstruction operator $\Dph : \dweber \to \Poly^p(\t_h)$ through its local restrictions $\DpT: \dweberT \to \Poly^p(T)$ to any mesh cell $T \in \t_h$: for all $\vTF\in\dweberT$, $\DpT(\vTF)\in\Poly^p(T)$ is the unique solution to 
\begin{equation}\label{def:DTk}
  \big(\DpT(\vTF), q\big)_T = -(\eta_T\vT,\Grad q)_T   + \!\!\sum_{F\in\f_T}\!\varepsilon_{T,F}\big(\vFn, q_{\mid F}\big)_F \qquad\forall q \in\Poly^p(T).
\end{equation}
Performing an integration by parts of the term $(\eta_T\vT,\Grad q)_T$, and using a discrete trace inequality, it is an easy matter to prove that
\begin{equation} \label{eq:bound.d}
  \sum_{T\in\t_h}\left(\|\DpT(\vTF)\|_{0,T}^2+s_{\Div,T}(\vTF,\vTF)\right)\lesssim|\vh|_{\Div,h}^2\qquad\forall\vh\in\dweber.
\end{equation}
If $p\geq\ell-1$, one can also prove that $|\vh|_{\Div,h}^2\lesssim\sum_{T\in\t_h}\left(\|\DpT(\vTF)\|_{0,T}^2+s_{\Div,T}(\vTF,\vTF)\right)$.

\subsection{Compactness in $\underline{\mathbbmsl{X}}_{h,\vec{\tau}}^{\ell}$}

Recall the definitions~\eqref{eq:dweber} of the discrete space $\dfweber$, and~\eqref{eq:senocu}--\eqref{eq:senodi} of the hybrid semi-norms $|{\cdot}|_{\Curl,h}$ and $|{\cdot}|_{\Div,h}$. Leveraging the first hybrid Weber inequality from Theorem~\ref{thm:fweber}, we infer that
\begin{equation} \label{eq:nocu}
  \|\vh\|_{\vec{X}_{\vec{\tau}},h}^2\defi \eta_\flat^{-1}|\vh|_{\Div,h}^2+\eta_\sharp|\vh|_{\Curl,h}^2+\eta_\flat^{-1}\sum_{j=1}^{\beta_2}\big|\big(\gamma_{n,h}(\vh),1\big)_{\Gamma_j}\big|^2
\end{equation}
defines a norm on $\dfweber$, which can be proven uniformly (in $h$) equivalent to the $\|{\cdot}\|_{\vec{X},h}$-norm~\eqref{eq:no}: there is $c_{\vec{\tau}}>0$, independent of $h$ and $\eta$, such that
$$(c_{\vec{\tau}}\kappa_\eta)^{-1}\|\vh\|_{\vec{X},h}^2\leq\|\vh\|_{\vec{X}_{\vec{\tau}},h}^2\leq c_{\vec{\tau}}\kappa_\eta\|\vh\|_{\vec{X},h}^2\qquad\forall\vh\in\dfweber.$$

\begin{theorem}[Maxwell compactness in $\dfweber$] \label{th:max1}
  Let $\big(\mathcal{M}_h\big)_{h\in\Hindex}$ be a regular mesh family as defined in Section~\ref{sse:pol.mesh}. Let $(\vh)_{h\in\Hindex}$ be a sequence of elements of $\dfweber$ for which there exists a real number $c_M>0$ (independent of $h$) such that $\eta_\sharp^{-\frac12}\|\vh\|_{\vec{X}_{\vec{\tau}},h}\leq c_M$ for all $h\in\Hindex$.
  Then, there exists an element $\vec{v}\in\cfweber$ such that, along a subsequence, as $h \to 0$,
  \begin{enumerate}[label=(\alph*),leftmargin=2\parindent]
    \item $\vTh \rightarrow \vec{v}$ strongly in $\bLL$;
    \item $\Cmh(\vh)\rightharpoonup\Curl\vec{v}$ weakly in $\bLL$;
    \item $\Dph(\vh)\rightharpoonup\Div(\eta\vec{v})$ weakly in $\LL$;
    \item $\gamma_{n,h}(\vh)\rightharpoonup(\eta\vec{v})_{\mid\Gamma}{\cdot}\normal$ weakly in $H^{-\frac12}(\Gamma)$.
  \end{enumerate}
\end{theorem}
\begin{proof}
  As in the continuous case (cf.~Proposition~\ref{pr:max1}), the proof proceeds in three steps.
  
  \smallskip
  {\em Step 1 (Weak convergence):} By assumption, $\|\vh\|_{\vec{X}_{\vec{\tau}},h}\leq c_M\eta_\sharp^{\frac12}$ for all $h\in\Hindex$. By definition~\eqref{eq:nocu} of the $\|{\cdot}\|_{\vec{X}_{\vec{\tau}},h}$-norm, by the boundedness results~\eqref{eq:bound.c} and~\eqref{eq:bound.d}, and by the first hybrid Weber inequality~\eqref{thm:max.ineq.eq1}, this implies that (i) $(\vTh)_{h\in\Hindex}$ and $\big(\Cmh(\vh)\big)_{h\in\Hindex}$ are uniformly bounded in $\bLL$, and (ii) $\big(\Dph(\vh)\big)_{h\in\Hindex}$ is uniformly bounded in $\LL$. By weak compactness, we then infer the existence of $\vec{v}\in\bLL$, $\vec{c}\in\bLL$, and $d\in\LL$ such that, along a subsequence (not relabelled), (i) $\vTh\rightharpoonup\vec{v}$ and $\Cmh(\vh)\rightharpoonup\vec{c}$ weakly in $\bLL$, and (ii) $\Dph(\vh)\rightharpoonup d$ weakly in $\LL$.

  Let us first prove that $\vec{v}\in\Hzcurl$ and that $\vec{c}=\Curl\vec{v}$. Let $\vec{z}\in\vec{C}^\infty(\overline{\Omega})$. There holds
  \begin{equation} \label{eq:parts}
    \begin{split}
      \big(\Cmh (\vh), \vec{z} \big)_\Omega
      &=\sum_{T\in\t_h}\big(\CmT (\vTF),\vec{\pi}_T^m(\vec{z}_{\mid T})\big)_T\\
      &=\sum_{T\in\t_h} \bigg(\big( \Curl \vT, \vec{\pi}_T^m(\vec{z}_{\mid T}) \big)_T\\
      &\qquad+ \sum_{F\in\f_T}\varepsilon_{T,F}\big(\vec{v}_{T\mid F}{\times}\normal_F-\vFt,\normal_F{\times}(\vec{\pi}^m_T(\vec{z}_{\mid T})_{\mid F}{\times}\normal_F)\big)_F \bigg)
      \\
      &=\big(\Curl_h \vTh, \vec{z}\big)_{\Omega}
      +\sum_{T \in \t_h} \sum_{F\in\f_T}\varepsilon_{T,F}\big(\vec{v}_{T\mid F}{\times}\normal_F-\vFt,\normal_F{\times}(\vec{z}_{\mid F}{\times}\normal_F)\big)_F+\mathcal{I}_{\Curl,h}
      \\
      &=\big(\vTh, \Curl\vec{z}\big)_{\Omega}+\mathcal{I}_{\Curl,h},
    \end{split}
  \end{equation}
  with $\mathcal{I}_{\Curl,h}$ given by
  \begin{multline*}
    \mathcal{I}_{\Curl,h}\defi\big(\Curl_h \vTh,\vec{\pi}^m_h(\vec{z})-\vec{z}\big)_{\Omega}\\+\sum_{T \in \t_h} \sum_{F\in\f_T}\varepsilon_{T,F}\big(\vec{v}_{T\mid F}{\times}\normal_F-\vFt,\normal_F{\times}((\vec{\pi}^m_T(\vec{z}_{\mid T})-\vec{z})_{\mid F}{\times}\normal_F)\big)_F,
  \end{multline*}                        
  where we have (i) used that $\CmT(\vTF)\in\bPoly^m(T)$ for all $T\in\t_h$ in the first line, (ii) applied the definition~\eqref{def:CTk} of $\CmT(\vTF)$ along with an integration by parts of the term $\big(\vT,\Curl\vec{\pi}^m_T(\vec{z}_{\mid T})\big)_T$ in the second line, (iii) added/subtracted $\big(\Curl_h \vTh, \vec{z}\big)_{\Omega}$ and $\sum_{T \in \t_h} \sum_{F\in\f_T}\varepsilon_{T,F}\big(\vec{v}_{T\mid F}{\times}\normal_F-\vFt,\normal_F{\times}(\vec{z}_{\mid F}{\times}\normal_F)\big)_F$ in the third line, and (iv) performed cell-by-cell integration by parts of the term $\big(\Curl_h \vTh, \vec{z}\big)_{\Omega}$ and used that $\sum_{T \in \t_h} \sum_{F\in\f_T}\varepsilon_{T,F}\big(\vFt,\normal_F{\times}(\vec{z}_{\mid F}{\times}\normal_F)\big)_F=0$ (since $\vec{z}$ is single-valued at interfaces and $\vFt\equiv\vec{0}$ for all $F\in\f_h^{\partial}$ for $\vh\in\dfweber$) in the fourth line. For the term $\mathcal{I}_{\Curl,h}$ in the right-hand side of~\eqref{eq:parts}, standard approximation properties for $\vec{\pi}_h^m$ on mesh cells (cf.~e.g.~\cite[Thm.~1.45]{DPDro:20}) on the one side, and repeated applications of the triangle and Cauchy--Schwarz inequalities along with the estimate~\eqref{eq:withoutproj} and approximation properties for $\vec{\pi}_h^m$ on mesh faces on the other side, yield
  \begin{equation*}
    |\mathcal{I}_{\Curl,h}|\lesssim h^{m+1}|\vh|_{\Curl,h}|\vec{z}|_{m+1}.
  \end{equation*}
  Starting from~\eqref{eq:parts}, and using that $\vTh\rightharpoonup\vec{v}$ and $\Cmh(\vh)\rightharpoonup\vec{c}$ weakly in $\bLL$, combined with the fact that $\mathcal{I}_{\Curl,h}\to 0$ as $h\to 0$ (since $|\vh|_{\Curl,h}\leq c_M$ for all $h\in\Hindex$ and $m+1\geq 1$), we infer that
  $$(\vec{c},\vec{z})_{\Omega}=(\vec{v},\Curl\vec{z})_{\Omega}\qquad\forall\vec{z}\in\vec{C}^\infty(\overline{\Omega}).$$
  Since the latter relation is, in particular, valid for all $\vec{z}\in\vec{C}^\infty_{\vec{0}}(\Omega)$, we infer that $\vec{v}\in\Hcurl$ and that $\vec{c}=\Curl\vec{v}$.
  Then, integrating by parts the left-hand side of the relation, we get that $\langle\vec{v}_{\mid\Gamma}{\times}\normal,\normal{\times}(\vec{z}_{\mid\Gamma}{\times}\normal)\rangle_{\Gamma}=0$ for all $\vec{z}\in\vec{C}^\infty(\overline{\Omega})$, which implies that $\vec{v}\in\Hzcurl$.

  Let us now prove that $\vec{v}\in\Hdiv$ and that $d=\Div(\eta\vec{v})$, along with the weak convergence result $\gamma_{n,h}(\vh)\rightharpoonup(\eta\vec{v})_{\mid\Gamma}{\cdot}\normal$ in $H^{-\frac12}(\Gamma)$. Let $z\in C^\infty(\overline{\Omega})$. Adapting the arguments used in~\eqref{eq:parts} in the case of the rotational operator to the case of the divergence operator, we infer that
  \begin{equation} \label{eq:parts2}
    \big(\Dph(\vh), z \big)_\Omega=-\big(\eta\,\vTh, \Grad z \big)_{\Omega}+\big(\gamma_{n,h}(\vh),z_{\mid\Gamma}\big)_{\Gamma}+\mathcal{I}_{\Div,h},
  \end{equation}
  where $\mathcal{I}_{\Div,h}$ is given by
  $$\mathcal{I}_{\Div,h}\defi\big(\Div_h(\eta\vTh),\pi^p_h(z)-z\big)_{\Omega}+\sum_{T \in \t_h} \sum_{F\in\f_T}\varepsilon_{T,F} \big(\vFn - \eta_T\vec{v}_{T\mid F}{\cdot}\normal_F, (\pi_T^p(z_{\mid T}) - z)_{\mid F} \big)_F.$$
  Adapting the arguments from the rotational case, it is an easy matter to prove that
  \begin{equation*}
    |\mathcal{I}_{\Div,h}|\lesssim h^{p+1}|\vh|_{\Div,h}|z|_{p+1}.
  \end{equation*}
  Let us first choose $z\in C^\infty_0(\Omega)$. Remark that the boundary term in~\eqref{eq:parts2} vanishes. Starting from~\eqref{eq:parts2}, and using that $\vTh\rightharpoonup\vec{v}$ weakly in $\bLL$ and $\Dph(\vh)\rightharpoonup d$ weakly in $\LL$, combined with the fact that $\mathcal{I}_{\Div,h}\to 0$ as $h\to 0$ (since $|\vh|_{\Div,h}\leq c_M\eta_\sharp$ for all $h\in\Hindex$ and $p+1\geq 1$), we infer that
  $$(d,z)_{\Omega}=-(\eta\,\vec{v},\Grad z)_{\Omega}\qquad\forall z\in C^\infty_0(\Omega).$$
  The latter relation directly implies that $\vec{v}\in\Hdiv$ and that $d=\Div(\eta\vec{v})$.
  Now, passing to the limit $h\to 0$ in~\eqref{eq:parts2} for a generic $z\in C^\infty(\overline{\Omega})$, we infer that
  $$\lim_{h\to 0}\big(\gamma_{n,h}(\vh),z_{\mid\Gamma}\big)_{\Gamma}=(\Div(\eta\vec{v}),z)_{\Omega}+(\eta\,\vec{v},\Grad z)_{\Omega}=\langle(\eta\vec{v})_{\mid\Gamma}{\cdot}\normal,z_{\mid\Gamma}\rangle_{\Gamma},$$
  which implies that $\gamma_{n,h}(\vh)\rightharpoonup(\eta\vec{v})_{\mid\Gamma}{\cdot}\normal$ weakly in $H^{-\frac12}(\Gamma)$.

  \smallskip
  {\em Step 2 (Characterization of the limit):} This step is identical to Step 2 from the continuous case (cf.~the proof of Proposition~\ref{pr:max1}), up to the replacement of $m\in\Natu$ by $h\in\Hindex$ (and $m\to\infty$ by $h\to 0$) and of $\vec{v}_m\in\cfweber$ by $\vTh\in\bPoly^{\ell}(\t_h)$. For any $h\in\Hindex$, by~\eqref{eq:helm1}, there exist $\varphi_h\in H^1_0(\Omega)$, $\vec{\psi}_h\in\vec{H}^1(\Omega)\cap\vec{L}^2_{\vec{0}}(\Omega)$, and $\vec{w}_h\in\fharmonic$, such that
  \begin{equation} \label{eq:hodge2}
    \vec{v}_h=\Grad\varphi_h+\frac{\eta_\sharp}{\eta}\Curl\vec{\psi}_h+\vec{w}_h,
  \end{equation}
  with $\|\vec{\psi}_h\|_1\leq C_{\Omega,1}\|\Curl\vec{\psi}_h\|_0$ by~\eqref{hh.reg1}.
  Proceeding as in the continuous case, one can infer the existence of $\vec{w}\in\fharmonic$ such that, up to extraction (not relabelled), $\vec{w}_h\to\vec{w}$ strongly in $\bLL$, and of $\varphi\in H^1_0(\Omega)$ and $\vec{\psi}\in\vec{H}^1(\Omega)\cap\vec{L}^2_{\vec{0}}(\Omega)$ such that, up to extractions (not relabelled), $\Grad\varphi_h\rightharpoonup\Grad\varphi$ and $\Curl\vec{\psi}_h\rightharpoonup\Curl\vec{\psi}$ weakly in $\bLL$. By linearity, we have thus proven that, along a subsequence (not relabelled), $\vec{v}_h\rightharpoonup\Grad\varphi+\frac{\eta_\sharp}{\eta}\Curl\vec{\psi}+\vec{w}$ weakly in $\bLL$. By uniqueness of the weak limit, $\vec{v}=\Grad\varphi+\frac{\eta_\sharp}{\eta}\Curl\vec{\psi}+\vec{w}$.

  \smallskip
  {\em Step 3 (Strong convergence):} Here again, the arguments are inspired by what is done in Step 3 from the continuous case, but with the important difference that here $\vTh\in\bPoly^{\ell}(\t_h)\not\subset\cfweber$. By Rellich's compactness theorem, we actually have that, along the same subsequence (not relabelled) as in Step 2 above, $\varphi_h\to\varphi$ strongly in $\LL$ and $\vec{\psi}_h\to\vec{\psi}$ strongly in $\bLL$.
  We now want to prove the strong convergences of $(\Grad\varphi_h)_{h\in\Hindex}$ and $(\Curl\vec{\psi}_h)_{h\in\Hindex}$ in $\bLL$ (we remind the reader that we already proved strong convergence for $(\vec{w}_h)_{h\in\Hindex}$).

  Recalling the expression of $\vec{v}\in\cfweber$ derived in Step 2, since the decomposition~\eqref{eq:hodge2} is $\vec{L}^2_\eta(\Omega)$-orthogonal, there holds
  \begin{equation*}
    \begin{split}
      \|\eta^{\frac12}(\Grad\varphi_h-\Grad\varphi)\|_0^2&=\big(\eta(\vec{v}_h-\vec{v}),\Grad(\varphi_h-\varphi)\big)_{\Omega}\\
      &=\sum_{T\in\t_h}\bigg(-\big(\Div(\eta_T(\vT-\vec{v})),\varphi_h-\varphi\big)_T\\
      &\qquad+\sum_{F\in\f_T}\big(\eta_T\vec{v}_{T\mid F}{\cdot}\normal_{T,F},(\varphi_h-\varphi)_{\mid F}\big)_F\bigg)\\
      &=-\big(\Div_h(\eta(\vTh-\vec{v})),\varphi_h-\varphi\big)_{\Omega}\\
      &\qquad+\sum_{T\in\t_h}\sum_{F\in\f_T}\varepsilon_{T,F}\big(\eta_T\vec{v}_{T\mid F}{\cdot}\normal_F-\vFn,(\varphi_h-\varphi)_{\mid F}\big)_F,
    \end{split}
  \end{equation*}
  where, to pass from the first to the second line, we have performed cell-by-cell integration by parts and used that $\sum_{T\in\t_h}\big\langle(\eta\vec{v})_{\mid\partial T}{\cdot}\normal_{\partial T},(\varphi_h-\varphi)_{\mid\partial T}\big\rangle_{\partial T}=0$ (as a consequence of the fact that $\vec{v}\in\Hdiv$ and that $\varphi_h,\varphi\in H^1_0(\Omega)$) and, to pass from the second to the third line, we have again used the fact that $(\varphi_h-\varphi)$ is single-valued at interfaces and vanishes on the boundary of the domain to subtract the zero contribution $\sum_{T\in\t_h}\sum_{F\in\f_T}\varepsilon_{T,F}\big(\vFn,(\varphi_h-\varphi)_{\mid F}\big)_F$. By Cauchy--Schwarz inequality, combined with a continuous trace inequality (along with mesh regularity), and the fact that $|\vh|_{\Div,h}\leq c_M\eta_\sharp$ for all $h\in\Hindex$ and $\|\Div(\eta\vec{v})\|_0\lesssim\lim\inf_{h\to 0}|\vh|_{\Div,h}$ (by weak convergence of $(\Dph(\vh))_{h\in\Hindex}$ to $\Div(\eta\vec{v})$ in $\LL$ and the uniform bound~\eqref{eq:bound.d}), we infer that
  \begin{equation} \label{eq:expr}
    \|\eta^{\frac12}(\Grad\varphi_h-\Grad\varphi)\|_0^2\lesssim\eta_\sharp\big(\|\varphi_h-\varphi\|_0+h\|\Grad\varphi_h-\Grad\varphi\|_0\big).
  \end{equation}
  Since $\|\Grad\varphi_h\|_0\leq\eta_\flat^{-\frac12}\|\eta^{\frac12}\vTh\|_0$ (by $\vec{L}^2_\eta(\Omega)$-orthogonality of the decomposition~\eqref{eq:hodge2}), $\|\eta^{\frac12}\vTh\|_0\leq c_{W,1}\sqrt{3}\kappa_\eta^{\frac12}c_M\eta_\sharp^{\frac12}$ (by~\eqref{thm:max.ineq.eq1},~\eqref{eq:nocu},~and $\|\vh\|_{\vec{X}_{\vec{\tau}},h}\leq c_M\eta_\sharp^{\frac12}$), and $\|\Grad\varphi\|_0\leq\lim\inf_{h\to 0}\|\Grad\varphi_h\|_0$ (since $\Grad\varphi_h\rightharpoonup\Grad\varphi$ weakly in $\bLL$), we infer by the triangle inequality that
  $$\|\Grad\varphi_h-\Grad\varphi\|_0\leq 2c_{W,1}\sqrt{3}\kappa_\eta c_M.$$
  Combining this uniform boundedness result to the fact that $\varphi_h\to\varphi$ strongly in $\LL$ to pass to the limit $h\to 0$ in~\eqref{eq:expr}, we conclude that $\Grad\varphi_h\to\Grad\varphi$ strongly in $\bLL$.

  Adapting the latter arguments, there also holds
  \begin{equation*}
    \begin{split}
      \eta_\sharp\|\eta^{-\frac12}(\Curl\vec{\psi}_h-\Curl\vec{\psi})\|_0^2&=\big(\vec{v}_h-\vec{v},\Curl(\vec{\psi}_h-\vec{\psi})\big)_{\Omega}\\
      &=\big(\Curl_h(\vec{v}_h-\vec{v}),\vec{\psi}_h-\vec{\psi}\big)_{\Omega}\\
      &\qquad+\sum_{T\in\t_h}\sum_{F\in\f_T}\varepsilon_{T,F}\big(\vec{v}_{T\mid F}{\times}\normal_F-\vFt,\normal_F{\times}((\vec{\psi}_h-\vec{\psi})_{\mid F}{\times}\normal_F)\big)_F,
    \end{split}
  \end{equation*}
  where we have used that $\vh\in\dfweber$, $\vec{v}\in\Hzcurl$, and $\vec{\psi}_h,\vec{\psi}\in\vec{H}^1(\Omega)$. Additionally invoking~\eqref{eq:withoutproj}, this implies that
  $$\eta_\sharp\|\eta^{-\frac12}(\Curl\vec{\psi}_h-\Curl\vec{\psi})\|_0^2\lesssim\big(\|\vec{\psi}_h-\vec{\psi}\|_0+h|\vec{\psi}_h-\vec{\psi}|_1\big),$$
  which eventually yields, since $|\vec{\psi}_h-\vec{\psi}|_1\leq 2C_{\Omega,1}c_{W,1}\sqrt{3}\kappa_\eta^{\frac12}c_M$ (here, we additionally used the estimate $|\vec{\psi}_h|_1\leq C_{\Omega,1}\|\Curl\vec{\psi}_h\|_0$ and the fact that $\vec{\psi}_h\rightharpoonup\vec{\psi}$ weakly in $\vec{H}^1(\Omega)$), and $\vec{\psi}_h\to\vec{\psi}$ strongly in $\bLL$, that $\Curl\vec{\psi}_h\to\Curl\vec{\psi}$ strongly in $\bLL$.

  Thus, along a subsequence (not relabelled), $\vec{v}_h\to\vec{v}$ strongly in $\bLL$, yielding the conclusion.
\end{proof}

\subsection{Compactness in $\underline{\mathbbmsl{X}}_{h,n}^{\ell}$}

Recall the definitions~\eqref{eq:dweber} of the discrete space $\dsweber$, and~\eqref{eq:senocu}--\eqref{eq:senodi} of the hybrid semi-norms $|{\cdot}|_{\Curl,h}$ and $|{\cdot}|_{\Div,h}$. Leveraging the second hybrid Weber inequality from Theorem~\ref{thm:sweber}, we infer that
\begin{equation} \label{eq:nodi}
  \|\vh\|_{\vec{X}_n,h}^2\defi\eta_\flat^{-1}|\vh|_{\Div,h}^2+\eta_\sharp|\vh|_{\Curl,h}^2+\eta_\flat^{-1}\sum_{i=1}^{\beta_1}\big|\big(\gamma_{n,h}(\vh),1\big)_{\Sigma_i}\big|^2
\end{equation}
defines a norm on $\dsweber$, which can be proven uniformly (in $h$) equivalent to the $\|{\cdot}\|_{\vec{X},h}$-norm~\eqref{eq:no}: there is $c_n>0$, independent of $h$ and $\eta$, such that
$$(c_n\kappa_\eta)^{-1}\|\vh\|_{\vec{X},h}^2\leq\|\vh\|_{\vec{X}_n,h}^2\leq c_n\kappa_\eta\|\vh\|_{\vec{X},h}^2\qquad\forall\vh\in\dsweber.$$

\begin{theorem}[Maxwell compactness in $\dsweber$] \label{th:max2}
  Let $\big(\mathcal{M}_h\big)_{h\in\Hindex}$ be a regular mesh family as defined in Section~\ref{sse:pol.mesh}. Let $(\vh)_{h\in\Hindex}$ be a sequence of elements of $\dsweber$ for which there exists a real number $c_M>0$ (independent of $h$) such that $\eta_\sharp^{-\frac12}\|\vh\|_{\vec{X}_n,h}\leq c_M$ for all $h\in\Hindex$.
  Then, there exists an element $\vec{v}\in\csweber$ such that, along a subsequence, as $h \to 0$,
  \begin{enumerate}[label=(\alph*),leftmargin=2\parindent]
    \item $\vTh \rightarrow \vec{v}$ strongly in $\bLL$;
    \item $\Cmh(\vh)\rightharpoonup\Curl\vec{v}$ weakly in $\bLL$;
    \item $\Dph(\vh)\rightharpoonup\Div(\eta\vec{v})$ weakly in $\LL$;
    \item $\vec{\gamma}_{\vec{\tau},h}(\vh)\rightharpoonup\vec{v}_{\mid\Gamma}{\times}\normal$ weakly in $\vec{H}^{-\frac12}(\Gamma)$.
  \end{enumerate}
\end{theorem}
\begin{proof}
  The proof proceeds in three steps, as in the continuous case (cf.~Proposition~\ref{pr:max2}). For the sake of brevity, since, up to inverting the roles of the rotational- and divergence-related parts in the proof, the three steps are essentially identical to those from Theorem~\ref{th:max1}, we do not detail them.
\end{proof}

\section*{Acknowledgments}
The work of the authors is supported by the action ``Pr\'eservation de l'emploi de R\&D'' from the ``Plan de Relance'' recovery plan of the French State.
Support from the LabEx CEMPI (ANR-11-LABX-0007) is also acknowledged.

\appendix

\section{Continuity modulus of $\vec{\Div}^{-1}$ on star-shaped cells} \label{ap:nim}

\begin{lemma}[Continuity modulus of $\vec{\Div}^{-1}$] \label{lemma:ineq.decomp}
  Let $\ell\in\Natu$. Let $T\in\t_h$, and assume that $T$ is star-shaped with respect to $\vec{x}_T$.
  Then, for all $d\in\Poly^{\ell-1}(T)$,
  \begin{equation} \label{lemma:ineq.decomp.eq}
    \|\vec{\Div}^{-1}d\|_{0,T} \leq \frac23 h_T\|d\|_{0,T}.
  \end{equation}
\end{lemma}
\begin{proof}
  The mapping $\Div:\kRT{\ell}\to\Poly^{\ell-1}(T)$ being an isomorphism, let $r\in\Poly^{\ell-1}(T)$ be the unique solution to
  $$d=\Div\big(r(\vec{x}-\vec{x}_T)\big).$$
  Using a simple vector calculus identity to expand the right-hand side, there holds
  $$d=3r+\Grad r{\cdot}(\vec{x}-\vec{x}_T).$$
  Now, multiplying both sides by $\overline{r}$, integrating over $T$, and taking the real part, we get
  \begin{align*}
    \mathfrak{R}\big((d,r)_T\big)
    &= 3 \|r\|_{0,T}^2 + \frac{1}{2} \big(\Grad (|r|^2),(\vec{x}-\vec{x}_T)\big)_T \\
    &= 3 \|r\|_{0,T}^2 - \frac{3}{2} \|r\|_{0,T}^2 + \frac12 \big(|r|^2,(\vec{x}-\vec{x}_T){\cdot}\normal_{\partial T}\big)_{\partial T} \\
    &\geq \frac32 \|r\|_{0,T}^2,
  \end{align*}
  where we have used an integration by parts formula to pass to the second line, and the fact that $T$ is star-shaped with respect to $\boldsymbol{x}_T$ to conclude. 
  Finally, applying the Cauchy--Schwarz inequality to $(d,r)_T$, we infer
  \begin{equation*}
    \|\vec{\Div}^{-1}d\|_{0,T}=\|r(\vec{x}-\vec{x}_T)\|_{0,T}\leq h_T\|r\|_{0,T} \leq \frac23 h_T \|d\|_{0,T},
  \end{equation*}
  which proves~\eqref{lemma:ineq.decomp.eq}.
\end{proof}

\ifJournal
\else

\bibliographystyle{plain}

\fi

{\small
\bibliography{hybrid_weber}
}

\end{document}